\documentclass[12pt,a4paper]{article}
\usepackage{amsmath}
\usepackage{makeidx}
\usepackage{amssymb}
\usepackage{amsthm}

\newtheorem{definition}{Definition}[section]
\newtheorem{proposition}{Proposition}[section]
\newtheorem{theorem}{Theorem}[section]

\newtheorem{corollary}{Corollary}[section]
\newtheorem{lemma}{Lemma}[section]

\begin{document}

\author{Nicoleta Aldea and Gheorghe Munteanu}
\title{Recent results on complex Cartan spaces}
\date{}
\maketitle

\begin{abstract}
In this paper, we first provide an updated survey of the geometry of complex
Cartan spaces. New characterizations for some particular classes of complex
Cartan spaces are pointed out, e.g. Landsberg-Cartan, strongly
Berwald-Cartan and others. We introduce the Cartan-Randers spaces which
offer examples of Berwald-Cartan and strongly Berwald-Cartan spaces. Then,
we investigate the complex geodesic curves of a complex Cartan space, using
the image by Legendre transformation ($\mathcal{L}-$ duality) of complex
geodesic curves of a complex Finsler space. Assuming the weakly K\"{a}hler
condition for a complex Cartan space, we establish that its complex geodesic
curves derive from Hamilton-Jacobi equations. Also, by $\mathcal{L}-$
duality, we introduce the corespondent notion of the projectively related
complex Finsler metrics, on the complex Cartan spaces. Various descriptions
of the projectively related complex Cartan metrics are given. As
applications, the projectiveness of a complex Cartan-Randers metric and the
locally projectively flat complex Cartan metrics are analyzed.
\end{abstract}

\begin{flushleft}
\bigskip \strut \textbf{2010 Mathematics Subject Classification:} 53B40,
53C60.

\end{flushleft}

\section{Introduction}

\setcounter{equation}{0}The study of geometry of the holomorphic cotangent
bundle, endowed with a complex Hamiltonian, has been deepened in some
previous works of the second author \cite{Mub, Mu1}. By analogy with the
real case, where remarkable results are known (\cite{M-S,An,Bw}), the
geometry achieved here is called complex Hamilton geometry. The particular
context in which the complex Hamiltonian is homogeneous on the fibre, is
known as complex Cartan geometry.

The approach of the complex Cartan spaces has been justified by the
existence of a pseudo-distance, on the dual holomorphic bundle, highlighted
by S. Kobayashi in \cite{Ko}. Using the equivalence method, J. Faran studied
in \cite{Fa} the complex Cartan spaces, (which he calls Finsler-Hamilton
spaces), with constant holomorphic curvature. He also gave some
characterizations of the dual Kobayashi metric.

On the other side, as we well know, the Hamiltonian Mechanics can be
retrieved via the Lagrangian Mechanics, by so called Legendre
transformation. This problem was already extended to the complex case (\cite%
{Mub}), the study of geometric objects on the holomorphic cotangent bundle $%
T^{\prime \ast }M,$ via the complex Legendre transformation, defined on the
holomorphic tangent bundle $T^{\prime }M,$ being called $\mathcal{L}-$ dual
process. By $\mathcal{L}-$ duality, it is shown that the dual Kobayashi
metric is exactly the $\mathcal{L}-$ dual of the well-known Kobayashi metric
on $T^{\prime }M$, (\cite{Mub, Mu1}).

The $\mathcal{L}-$ dual process seems a satisfactory technique for the
investigation of \ the geometry of complex Cartan spaces, using
corresponding notions from complex Finsler spaces, for which comprehensive
results are known, (\cite{A-P}-\cite{Al-Mu4},\cite{Mub}). But, a more
advertent analysis of previous results, obtained by $\mathcal{L}-$ dual
process, induced us to come back to some ideas which we reformulated then.
More exactly, the $\mathcal{L}-$ dual of the vertical natural frame $\frac{%
\partial }{\partial \eta ^{k}}$ on $T_{u}^{\prime }(T^{\prime }M)$ is
identified with the frame that is obtained by lifting the subscripts of the
vertical natural frame $\frac{\partial }{\partial \zeta _{k}}$ on $%
T_{u^{\ast }}^{\prime }(T^{\prime \ast }M),$ only in the purely Hermitian
case. This leads us to a lot of difficulties and some new ideas, which we
discuss and solve in the present paper.

\medskip

The paper is organized as follows. After a short survey of complex Cartan
spaces in our own notation, (Section 2), we extend some results about
classes of complex Cartan spaces obtained in \cite{Al-Mu5}. \ To the
Chern-Cartan complex nonlinear connection, with local coefficients $%
N_{ji}=-h_{j\bar{k}}\frac{\partial h^{\bar{k}l}}{\partial z^{i}}\zeta _{l}$,
we associate a complex linear connection of Berwald type $B\Gamma :=(N_{ji},
$ $B_{jk}^{i},$ $B_{\bar{j}k}^{\bar{\imath}},$ $0,$ $0),$ which is not of $%
(1,0)$ - type or metrical compatible. Here, we prove that the conditions: $%
B\Gamma $ is horizontal metrical compatible and $B\Gamma $ is of $(1,0)$ -
type are equivalent (Theorem 3.2) and, we call such a space
Landsberg-Cartan. Also, we obtain that any Landsberg-Cartan space with
weakly K\"{a}hler-Cartan property is a K\"{a}hler-Cartan space, (Theorem
3.3). The complex Berwald-Cartan spaces (i.e., the spaces with $%
B_{jk}^{i}(z) $) are Landsberg-Cartan. We show that any purely Hermitian
complex Cartan space is a complex Berwald-Cartan space. The complex
Berwald-Cartan spaces which are weakly K\"{a}hler-Cartan are called strongly
Berwald-Cartan spaces and they are contained in the class of K\"{a}%
hler-Cartan spaces, (Corollary 3.1). \ All these results are described in \
Section 3.

In Section 4, we introduce the Cartan-Randers metrics $\mathcal{\tilde{C}}%
=\alpha +|\beta |$, where $\alpha =\sqrt{a^{\bar{j}i}(z)\bar{\zeta}_{j}\zeta
_{i}}$ is a purely Hermitian complex Cartan metric on the complex manifold $%
M $ and $|\beta |$ is obtained by $\beta =b^{i}\zeta _{i},$ $b^{i}:=a^{\bar{j%
}i}(z)b_{\bar{j}}(z),$ with $b_{i}(z)$ the local coefficients of a
differential $(1,0)-$ form on $M.$ Complex Cartan-Randers metrics are
remarkable, they represent the medium in which Hermitian geometry properly
interferes with complex Cartan geometry. \ Theorem 4.2 and Corollary 4.1
report on the necessary and sufficient conditions for a complex
Cartan-Randers metric to be a Berwald-Cartan metric or strongly
Berwald-Cartan metric. The existence of complex Cartan-Randers spaces with
Berwald-Cartan and strongly Berwald-Cartan properties is attested by some
explicit examples.

The problem of the complex Cartan spaces obtained as image of the complex
Finsler spaces, via complex Legendre transformation is described in Section
5. First of all, we deduce the correct form of the $\mathcal{L}-$ dual of
vertical natural frame $\frac{\partial }{\partial \eta ^{k}},$ called the
nonholonomic vertical frame, (Theorems 5.1, 5.2). Also, we determine the $%
\mathcal{L}-$ dual of Chern-Finsler complex linear connection and by $%
\mathcal{L}-$ duality the weakly K\"{a}hler-Finsler property is sent in
weakly K\"{a}hler-Cartan property. The problem of the complex geodesic
curves of a complex Cartan spaces is also investigated by $\mathcal{L}-$
dual process. The image $\sigma ^{\ast }(s),$ by $\mathcal{L}-$ duality of
the complex geodesics curve $\sigma (s)$ of a complex Finsler spaces is
obtained, (Theorem 5.3) and, in the weakly K\"{a}hler-Cartan case $\sigma
^{\ast }(s)$ is a solution of the Hamilton-Jacobi equations. $\sigma ^{\ast
}(s)$ is called the complex geodesic curve of a complex Cartan space and its
equations can be rewritten in a more significant form as in Theorem 5.4. \

The projectively related complex Cartan spaces are approached by $\mathcal{L}%
-$ duality, too. Two complex Cartan metrics $\mathcal{\tilde{C}}$ and $%
\mathcal{C}$ on a common underlying manifold $M,$ obtained \ by $\mathcal{L}%
- $ duality, are called projectively related if any complex geodesic curve,
in the sense describe above, of the first is also a complex geodesic curve
for the second and vice versa. This means that between the functions $\tilde{%
N}_{k}$ and $N_{k}$ there is a so-called projective change $\tilde{N}%
_{k}=N_{k}+B_{k}+Q\zeta _{k}$, where $Q$ is a smooth function on $\widetilde{%
T^{\prime \ast }M}$ with complex values and $B_{k}:=\tilde{h}_{sk}\tilde{%
\Theta}^{\ast s}-h_{sk}\Theta ^{\ast s},$ (Theorem 5.5). Finally,
considering a Cartan-Randers metric $\mathcal{\tilde{C}}=\alpha +|\beta |$,
we prove that $\mathcal{\tilde{C}}$ can be the image by $\mathcal{L}-$
duality of a complex Finsler metric only if it is purely Hermitian. Then, we
find the necessary and sufficient conditions under which $\mathcal{\tilde{C}}
$ and $\alpha $ are projectively related. Also, the locally projectively
flat complex Cartan metrics are pointed out, (Corollary 5.3).

\section{Preliminaries}

\setcounter{equation}{0}Geometry of real Finsler spaces is already one
classic today, (\cite{A-P,B-S,Ma1,Sh1}, etc.). During the last years, we
remark a significant progress in the study of complex Finsler geometry, (%
\cite{A-P}-\cite{Al-Mu4},\cite{Mub,Wo,C-S}, etc.). Also, the study of Cartan
spaces (real and complex) is enthralling, (\cite{M-S, An, Mub}, etc.)

Let $M$ be a $n$ $-$ dimensional complex manifold and $z=(z^{k})_{k=%
\overline{1,n}}$ be complex coordinates in a local chart. The complexified
of the real tangent bundle $T_{C}M$ splits into the sum of holomorphic
tangent bundle $T^{\prime }M$ and its conjugate $T^{\prime \prime }M$. The
bundle $T^{\prime }M$ is itself a complex manifold and the coordinates in a
local chart will be denoted by $u=(z^{k},\eta ^{k})_{k=\overline{1,n}}.$
These are changed into $(z^{\prime k},\eta ^{\prime k})_{k=\overline{1,n}}$
by the rules $z^{\prime k}=z^{\prime k}(z)$ and $\eta ^{\prime k}=\frac{%
\partial z^{\prime k}}{\partial z^{l}}\eta ^{l},$ $rank(\frac{\partial
z^{\prime k}}{\partial z^{l}})=n.$ The dual of $T^{\prime }M$ is denoted by $%
T^{\prime \ast }M$. On the manifold $T^{\prime \ast }M$, a point $u^{\ast }$
is characterized by the coordinates $u^{\ast }=(z^{k},\zeta _{k})_{k=%
\overline{1,n}},$ and a change of these has the form $z^{\prime k}=z^{\prime
k}(z)$ and$\;\;\zeta _{k}^{\prime }=\frac{\partial ^{\ast }z^{j}}{\partial
z^{\prime k}}\zeta _{j}$, $rank(\frac{\partial ^{\ast }z^{\prime k}}{%
\partial z^{l}})=n$. Here and further, we use the notation with star for the
partial derivatives with respect to $z$, on $T^{\prime \ast }M,$ only to
distinguish them from those on $T^{\prime }M$.

\begin{definition}
A \textit{complex Cartan space} is a pair $(M,\mathcal{C})$, where $\mathcal{%
C}:T^{\prime \ast }M\rightarrow \mathbb{R}^{+}$ is a continuous function
satisfying the conditions:

\textit{i)} $H:=\mathcal{C}^{2}$ is smooth on $\widetilde{T^{\prime \ast }M}%
:=T^{\prime \ast }M\backslash \{0\};$

\textit{ii)} $\mathcal{C}(z,\zeta )\geq 0$, the equality holds if and only
if $\zeta =0;$

\textit{iii)} $\mathcal{C}(z,\lambda \zeta )=|\lambda |C(z,\zeta )$ for $%
\forall \lambda \in \mathbb{C}$;

\textit{iv)} the Hermitian matrix $\left( h^{\bar{j}i}(z,\zeta )\right) $ is
positive definite, where $h^{\bar{j}i}:=\frac{\partial ^{2}H}{\partial \zeta
_{i}\partial \bar{\zeta}_{j}}$ is the fundamental metric tensor.
\end{definition}

Equivalently, the condition \textit{iv)} means that the indicatrix is
strongly pseudo-convex.

Consequently, from $iii$) we have $\frac{\partial H}{\partial \zeta _{k}}%
\zeta _{k}=\frac{\partial H}{\partial \bar{\zeta}_{k}}\bar{\zeta}_{k}=H,$ $%
\frac{\partial h^{\bar{j}i}}{\partial \zeta _{k}}\zeta _{k}=\frac{\partial
h^{\bar{j}i}}{\partial \bar{\zeta}_{k}}\bar{\zeta}_{k}=0$ and $H=h^{\bar{j}i}%
\bar{\zeta}_{j}\zeta _{i}.$ An usual example of complex Cartan space is so
called \textit{purely Hermitian} complex Cartan space, this means that $h^{%
\bar{j}i}=h^{\bar{j}i}(z).$

We say that a function $f$ on $T^{\prime \ast }M$ is $(p,q)$-\textit{%
homogeneous} with respect to the coordinate $\zeta =(\zeta _{k})$ iff $%
f(z^{k},\lambda \zeta _{k})=\lambda ^{p}\bar{\lambda}^{q}f(z^{k},\zeta _{k})$%
, for any $\lambda \in \mathbf{C}$. For instance, $H:=\mathcal{C}^{2}$ is a $%
(1,1)$ - homogeneous function.

Roughly speaking, the geometry of a complex Cartan space consists in the
study of the geometric objects of the complex manifold $T^{\prime *}M$
endowed with the Hermitian metric structure defined by $h^{\bar{j}i}.$
Therefore, further the first step is the study of sections in the
complexified tangent bundle of $T^{\prime *}M$, which is decomposed in the
sum $T_C(T^{\prime *}M)=T^{\prime }(T^{\prime *}M)\oplus T^{\prime \prime
}(T^{\prime *}M)$.

Let $VT^{\prime \ast }M\subset T^{\prime }(T^{\prime \ast }M)$ be the
vertical bundle, which has the vertical distribution $V_{u^{\ast
}}(T^{\prime \ast }M)$, locally spanned by $\{\frac{\partial }{\partial
\zeta _{k}}\}.$ A complex nonlinear connection, briefly $(c.n.c.)$, on $%
T^{\prime \ast }M$ is a supplementary subbundle in $T^{\prime }(T^{\prime
\ast }M)\;$of $V(T^{\prime \ast }M)$ , i.e., $T^{\prime }(T^{\prime \ast
}M)=H(T^{\prime \ast }M)\oplus V(T^{\prime \ast }M).$ The horizontal
distribution $H_{u^{\ast }}(T^{\prime \ast }M)$ is locally spanned by $\{%
\frac{\delta ^{\ast }}{\delta z^{j}}\},$ where $\frac{\delta ^{\ast }}{%
\delta z^{k}}=\frac{\partial ^{\ast }}{\partial z^{k}}+N_{jk}\frac{\partial
}{\partial \zeta _{j}}$ and functions $N_{jk}$ are the coefficients of the $%
(c.n.c.)$ on $T^{\prime \ast }M.$ The pair $\{\delta _{k}^{\ast }:=\frac{%
\delta ^{\ast }}{\delta z^{k}},\,\dot{\partial}^{k}:=\frac{\partial }{%
\partial \zeta _{k}}\}$ will be called the adapted frame of the $(c.n.c.)$,
which obey the change rules $\delta _{k}^{\ast }=\frac{\partial ^{\ast
}z^{\prime j}}{\partial z^{k}}\delta _{j}^{\ast \prime }$ and $\dot{\partial}%
^{k}=\frac{\partial ^{\ast }z^{\prime k}}{\partial z^{j}}\dot{\partial}%
^{\prime j}.$ By conjugation everywhere we have obtained an adapted frame $%
\{\delta _{\bar{k}}^{\ast },\dot{\partial}^{\bar{k}}\}$ on $T_{u^{\ast
}}^{\prime \prime }(T^{\prime \ast }M).$ The dual adapted frames are $%
\{d^{\ast }z^{k},\;\delta \zeta _{k}=d\zeta _{k}-N_{kj}dz^{j}\}$ and $%
\{d^{\ast }\bar{z}^{k},\delta \bar{\zeta}_{k}\}.$

A Hermitian connection $D$, of $(1,0)-$ type is so called Chern-Cartan
connection (cf. \cite{Mub}), in brief $C-C$ connection, and it is locally
given by the following coefficients%
\begin{equation}
N_{ji}=-h_{j\bar{k}}\frac{\partial ^{\ast }h^{\bar{k}l}}{\partial z^{i}}%
\zeta _{l}\;;\;H_{jk}^{i}:=h^{\bar{m}i}(\delta _{k}^{\ast }h_{j\bar{m}%
})\;;\;V_{j}^{ik}:=-h_{j\bar{m}}(\dot{\partial}^{k}h^{\bar{m}i}),
\label{1.3}
\end{equation}%
and $H_{\bar{j}k}^{\bar{\imath}}=V_{\bar{j}}^{\bar{\imath}k}=0$, where here
and hereinafter $\delta _{k}^{\ast }$ is the adapted frame of the $C-C$ $%
(c.n.c.),$ $h_{j\bar{k}}h^{\bar{k}l}=\delta _{j}^{l}$ and $D_{\delta
_{k}^{\ast }}\delta _{j}^{\ast }=H_{jk}^{i}\delta _{i},$ $D_{\delta
_{k}^{\ast }}\dot{\partial}^{i}=-H_{jk}^{i}\dot{\partial}^{j},\;D_{\dot{%
\partial}^{k}}\delta _{j}^{\ast }$ $=V_{j}^{ik}\delta _{i}^{\ast },\;D_{\dot{%
\partial}^{k}}\dot{\partial}^{i}=-V_{j}^{ik}\dot{\partial}^{j},$ etc.
Moreover, we have%
\begin{equation}
H_{jk}^{i}=\dot{\partial}^{i}N_{jk}\;;\;H_{jk}^{i}\zeta _{i}=N_{jk}\;.
\label{1.3'}
\end{equation}%
Denoting by $"\shortmid "$ , $"\mid "$ , $"\overline{\shortmid }"$ and $"%
\bar{\mid}",$ the $h^{\ast }-$, $v^{\ast }-$, $\overline{h}^{\ast }-$, $%
\overline{v^{\ast }}-$ covariant derivatives with respect to $C-C$
connection, respectively, it results $h_{|k}^{\bar{j}i}=h_{|\bar{k}}^{\bar{j}%
i}=h^{\bar{j}i}|_{k}=h^{\bar{j}i}|_{\overline{k}}=0,$ i.e. $C-C$ connection
is $h^{\ast }-$ and $v^{\ast }-$ metrical.

For more details on complex Cartan spaces, see \cite{Mub}. Further on, in
order to simplify the writing, we use a bar over indices to denote the
complex conjugation of the variables or of the frames, e.g., $\zeta _{\bar{k}%
}:=\bar{\zeta}_{k}$ or $\dot{\partial}^{\bar{k}}:=\overline{\dot{\partial}%
^{k}}$.

\section{Complex Landsberg-Cartan spaces}

\setcounter{equation}{0}In \cite{Al-Mu5} we investigated some classes of
complex Cartan spaces. A complex Cartan space $(M,\mathcal{C})$ is called
\textit{strongly K\"{a}hler-Cartan} iff $T_{jk}^{\ast i}=0,$ \textit{K\"{a}%
hler-Cartan}$\;$iff $T_{jk}^{\ast i}\zeta _{i}=0$ and \textit{weakly K\"{a}%
hler-Cartan\ }if\textit{\ } $T_{jk}^{\ast i}\zeta _{i}\zeta ^{j}=0,$ where $%
T_{jk}^{\ast i}:=H_{jk}^{i}-H_{kj}^{i}$ is the $h-$torsion and $\zeta
^{j}:=h^{\bar{m}j}\zeta _{\bar{m}}.$ But, the notions of strongly K\"{a}hler
and K\"{a}hler coincide, as in complex Finsler geometry (\cite{C-S}). Also,
in the particular case of a purely Hermitian complex Cartan metric all those
nuances of K\"{a}hler-Cartan are same with $\frac{\partial h_{j\bar{m}}}{%
\partial z^{i}}=\frac{\partial h_{i\bar{m}}}{\partial z^{j}}.$ The space $(M,%
\mathcal{C})$ is called \textit{Berwald-Cartan} iff the coefficients $%
H_{jk}^{i}$ depend only on the position $z.$

\begin{theorem}
(\cite{Al-Mu5}) Let $(M,\mathcal{C})$ be a $n$ - dimensional complex Finsler
space. Then the following assertions are equivalent:

i) $(M,\mathcal{C})$ is a Berwald-Cartan;

ii) $N_{ji}$ are holomorphic in $\zeta $;

iii) $V_{|j}^{\bar{m}ik}=0;$

iv) $V_{|\bar{j}}^{\bar{m}ik}=0.$
\end{theorem}

Examples of complex Berwald-Cartan metrics are provided first by the class
of purely Hermitian complex Cartan metrics. Indeed, in this case the local
coefficients of \ $C-C$ $(c.n.c.),$ $N_{ji}=-h_{j\bar{k}}(z)\frac{\partial
^{\ast }h^{\bar{k}l}(z)}{\partial z^{i}}\zeta _{l}$ are holomorphic in $%
\zeta .$

Another example of Berwald-Cartan metric is given by the function%
\begin{equation}
\mathcal{C}^{2}=H(z,w;\zeta ,\upsilon ):=e^{2\sigma }\left( |\zeta
|^{4}+|\upsilon |^{4}\right) ^{\frac{1}{2}},\;\mbox{with}\;\;\zeta ,\upsilon
\neq 0,\;\;  \label{IV.9}
\end{equation}%
on $\mathbf{C}^{2}$, where $\sigma (z,w)$ is a real valued function and $%
|\zeta _{i}|^{2}:=\zeta _{i}\zeta _{\bar{\imath}},$ $\zeta _{i}\in \{\zeta
,\upsilon \}$, $i=1,2.$ In (\ref{IV.9}) we relabeled the usual local
coordinates $z^{1},$ $z^{2},$ $\zeta _{1},$ $\zeta _{2}$ as $z,$ $w,$ $\zeta
,$ $\upsilon ,$ respectively. Direct computation leads to%
\begin{equation*}
N_{11}=-\frac{\partial \sigma }{\partial z}\zeta ;\;N_{12}=-\frac{\partial
\sigma }{\partial w}\zeta ;\;N_{21}=-\frac{\partial \sigma }{\partial z}%
\upsilon ;\;N_{22}=-\frac{\partial \sigma }{\partial w}\upsilon ,
\end{equation*}%
which attest that $N_{ji}$, $i,j=i=1,2,$ are holomorphic in $\zeta $ and $%
\upsilon .$

\medskip Two complex nonlinear connections (Chern-Finsler and the canonical)
play a significant role in complex Finsler geometry, (see \cite%
{Al-Mu3,Al-Mu2}). The first induces a complex spray and, the derivative of
the local coefficients of this spray leads to the second. It would be
expected (taking into account the $\mathcal{L}-$ dual process which will be
described later) that similar things happen in complex Cartan geometry. So,
for a complex Cartan space, we can take the canonical $(c.n.c.)$ besides the
$C-C$ $(c.n.c.)$ as follows. The local coefficients of the canonical $%
(c.n.c.)$ on $(M,C)$ are defined by $\overset{c}{N_{ji}}:=(\dot{\partial}%
^{k}N_{ji})\zeta _{k}.$ But, the $(1,0)$ - homogeneity with respect to the
variables $\zeta =(\zeta _{k})$ and $\bar{\zeta}=(\zeta _{\bar{k}})$ of $%
N_{ji}$, (i.e., $(\dot{\partial}^{k}N_{ji})\zeta _{k}=N_{ji}$ and $(\dot{%
\partial}^{\bar{k}}N_{ji})\zeta_{\bar{k}}=0)$, implies $\overset{c}{N_{ji}}%
=N_{ji}.$ Therefore, in complex Cartan geometry only the $C-C$ $(c.n.c.)$
from (\ref{1.3}) is available.

Nevertheless, we associate to the $C-C$ $(c.n.c.)$ another complex linear
connection of Berwald type%
\begin{equation*}
B\Gamma :=\left( N_{ji},B_{jk}^{i}:=\dot{\partial}^{i}N_{jk},B_{\bar{j}k}^{%
\bar{\imath}}:=(\dot{\partial}^{\bar{\imath}}N_{ks})\zeta ^{s}\zeta _{\bar{j}%
},0,0\right) .
\end{equation*}%
$B\Gamma $ is neither $h^{\ast }-$ nor $v^{\ast }-$ metrical, (for more
details see \cite{Mub}). Moreover, we have the following properties:
\begin{equation*}
B_{jk}^{i}=H_{jk}^{i};\;(\dot{\partial}^{m}H_{jk}^{i})\zeta _{m}=0;\;(\dot{%
\partial}^{\bar{m}}H_{jk}^{i})\zeta _{\bar{m}}=0;\;(\dot{\partial}^{\bar{%
\imath}}N_{ks})\zeta ^{k}=0,
\end{equation*}%
and their conjugations.

\begin{definition}
Let $(M,\mathcal{C})$ be a $n$ - dimensional complex Cartan space. $(M,%
\mathcal{C})$ is called complex Landsberg-Cartan space if $B\Gamma $ is $%
h^{\ast }-$ metrical, (i.e., $h_{||k}^{\bar{j}i}=h_{||\bar{k}}^{\bar{j}i}=0,$
where $"\shortmid \shortmid "$ is $h^{\ast }-$covariant derivatives with
respect to $B\Gamma $).
\end{definition}

\begin{theorem}
Let $(M,\mathcal{C})$ be a $n$ - dimensional complex Cartan space. Then $(M,%
\mathcal{C})$ is a complex Landsberg-Cartan space if and only if $B\Gamma $
is of $(1,0)$ - type.
\end{theorem}

\begin{proof} We know that $h^{\ast }-$ covariant derivatives with respect to $%
B\Gamma $ of $h^{\bar{j}i}$ is $h_{||k}^{\bar{j}i}=\delta _{k}h^{\bar{j}%
i}-h^{\bar{j}l}B_{lk}^{i}-h^{\bar{m}i}B_{\bar{m}k}^{\bar{j}}.$ Since $%
B_{jk}^{i}=H_{jk}^{i}$ and $C-C$ connection is metrical, then $h_{||k}^{\bar{%
j}i}=-h^{\bar{m}i}B_{\bar{m}k}^{\bar{j}}.$ Thus, $h_{||k}^{\bar{j}i}=0$
if and only if $B_{\bar{m}k}^{\bar{j}}=0$, i.e., $B\Gamma $ is of $%
(1,0)$ - type.
\end{proof}

Note that any complex Berwald-Cartan space is a complex Landsberg-Cartan
space.

\begin{theorem}
Let $(M,\mathcal{C})$ be a $n$ - dimensional complex Cartan space. $(M,%
\mathcal{C})$ is a K\"{a}hler-Cartan space if and only if it is a
Landsberg-Cartan space with weakly K\"{a}hler-Cartan property.
\end{theorem}

\begin{proof} Since $(\dot{\partial}^{\bar{\imath}}N_{ks})\zeta ^{k}=0$, it is easy
to check that%
\begin{equation}
(\dot{\partial}^{\bar{\imath}}N_{ks})\zeta ^{s}=\dot{\partial}^{\bar{\imath}%
}[(N_{ks}-N_{sk})\zeta ^{s}]-(N_{ks}-N_{sk})h^{\bar{\imath}s}.  \label{IV.10}
\end{equation}

First, we suppose that $\mathcal{C}$ is a K\"{a}hler-Cartan metric. Then, (%
\ref{IV.10}) implies $(\dot{\partial}^{\bar{\imath}}N_{ks})\zeta ^{s}=0,$
and so $B_{\bar{m}k}^{\bar{j}}=0.$

Conversely, if $\mathcal{C}$ is a Landsberg-Cartan metric with weakly K\"{a}%
hler-Cartan property, then (\ref{IV.10}) becomes $(N_{ks}-N_{ks})h^{\bar{%
\imath}s}=0,$ which gives $N_{ks}=N_{ks}.$ This completes the proof.
\end{proof}

\begin{definition}
Let $(M,\mathcal{C})$ be a $n$ - dimensional complex Cartan space. $(M,%
\mathcal{C})$ is called strongly Berwald-Cartan space if it is weakly K\"{a}%
hler-Cartan and $H_{jk}^{i}(z).$
\end{definition}

Obviously, the strongly Berwald-Cartan spaces define a subclass of the
Berwald-Cartan spaces. Moreover, we can prove.

\begin{corollary}
Let $(M,\mathcal{C})$ be a $n$ - dimensional complex Cartan space. If $(M,%
\mathcal{C})$ is strongly Berwald-Cartan then it is a K\"{a}hler-Cartan
space.
\end{corollary}

\begin{proof}
Due to the weakly
K\"{a}hler-Cartan property, we have $(N_{jk}-N_{kj})\zeta ^j=0$. Now, differentiating this with respect to $\zeta _{\bar{m}},$ and
using $\dot{\partial}^{\bar{m}}N_{jk}=0,$ it results $(N_{jk}-N_{kj})h^{\bar{%
m}j}=0.$ From here, we obtain $N_{jk}=N_{kj},$ i.e., the space is
K\"{a}hler-Cartan.
\end{proof}

\section{Complex Cartan-Randers metrics}

\setcounter{equation}{0}On the complex manifold $M$ we consider $a:=a_{i\bar{%
j}}(z)d^{\ast }z^{i}\otimes d^{\ast }\bar{z}^{j}$ a Hermitian positive
metric with its inverse $a^{\bar{j}i}(z)$ and $b=b_{i}(z)d^{\ast }z^{i}$ a
differential $(1,0)-$ form. Then, $\alpha ^{2}(z,\zeta ):=a^{\bar{j}%
i}(z)\zeta_{\bar{j}}\zeta _{i},$ ($\alpha :T^{\prime \ast }M\rightarrow
\mathbb{R}^{+}),$ defines a purely Hermitian complex Cartan metric on $M$.
Also, denoting by $b^{i}:=a^{\bar{j}i}(z)b_{\bar{j}},$ we construct a
function $\mathcal{\tilde{C}}:T^{\prime \ast }M\rightarrow \mathbb{R}^{+},$%
\begin{equation}
\mathcal{\tilde{C}}(z,\zeta ):=\alpha +|\beta |,  \label{II.1}
\end{equation}%
where%
\begin{equation}
|\beta |^{2}=\beta \overline{\beta }\;\mbox{
with   }\;\beta (z,\zeta )=b^{i}(z)\zeta _{i}.  \notag
\end{equation}%
A direct computation gives%
\begin{equation}
\tilde{h}^{\bar{j}i}:=\frac{\partial ^{2}\mathcal{\tilde{C}}^{2}}{\partial
\zeta _{i}\partial \bar{\zeta}_{j}}=\frac{\mathcal{\tilde{C}}}{\alpha }a^{%
\bar{j}i}-\frac{\mathcal{\tilde{C}}}{2\alpha ^{3}}\zeta ^{\bar{j}}\zeta ^{i}+%
\frac{\mathcal{\tilde{C}}}{2|\beta |}b^{\bar{j}}b^{i}+\frac{1}{2\mathcal{%
\tilde{C}}^{2}}\tilde{\zeta}^{\bar{j}}\tilde{\zeta}^{i},  \label{II.1''}
\end{equation}%
where%
\begin{equation*}
\zeta ^{i}:=\dot{\partial}^{i}\alpha ^{2}\;\;;\;\;\tilde{\zeta}^{i}:=\frac{%
\mathcal{\tilde{C}}}{\alpha }\zeta ^{i}+\frac{\mathcal{\tilde{C}}\overline{%
\beta }}{|\beta |}b^{i}.
\end{equation*}

Applying Proposition 2.2, from \cite{Al-Mu4}, it results.

\begin{proposition}
Corresponding to the function (\ref{II.1}), we have

i) $\displaystyle\tilde{h}_{i\bar{j}}=\frac{\alpha }{\mathcal{\tilde{C}}}a_{i%
\bar{j}}+\frac{|\beta |(\alpha ||b||^{2}+|\beta |)}{\mathcal{\tilde{C}}%
^{2}\gamma }\zeta _{i}\zeta _{\bar{j}}-\frac{\alpha ^{3}}{\mathcal{\tilde{C}}%
\gamma }b_{i}b_{\bar{j}}-\frac{\alpha }{\mathcal{\tilde{C}}\gamma }(\bar{%
\beta}\zeta _{i}b_{\bar{j}}+\beta b_{i}\zeta _{\bar{j}});$

ii) $\displaystyle\det \left( \tilde{h}^{\bar{j}i}\right) =\left( \frac{%
\mathcal{\tilde{C}}}{\alpha }\right) ^{n}\frac{\gamma }{2\alpha |\beta |}%
\det \left( a^{\bar{j}i}\right) ,$

where $||b||^{2}:=a^{\bar{j}i}b_{i}b_{\bar{j}}\;;\;\gamma :=\mathcal{\tilde{C%
}}^{2}+\alpha ^{2}(||b||^{2}-1).$
\end{proposition}

Having formula for $\det \left( \tilde{h}^{\bar{j}i}\right) $, we can say
that $\tilde{h}^{\bar{j}i}$ is positive definite if and only if $\gamma >0,$
for any $\zeta \in \widetilde{T^{\prime \ast }M}.$ Also, it is obvious that
the function $\mathcal{\tilde{C}}=\alpha +|\beta |$ satisfies the conditions
$i)-iii)$ from Definition 2.1. So, we have proved the following result.

\begin{theorem}
The function (\ref{II.1}) with $\gamma >0$ is a complex Cartan metric.
\end{theorem}

Further on, the function (\ref{II.1}) with $\gamma >0$ \ is called a \textit{%
complex Cartan-Randers metric} and the pair $(M,\alpha +|\beta |)$ is a
\textit{complex Cartan-Randers space}.

Note that a complex Cartan-Randers metric can be purely Hermitian if and
only if $\alpha ^{2}||b||^{2}=|\beta |^{2}$. Since any purely Hermitian
metric is a complex Berwald-Cartan metric, we now focus on non-purely
Hermitian complex Cartan-Randers metrics.

\medskip Once obtained the metric tensor of a complex Cartan-Randers space,
it is a technical computation to get the expression of the coefficients $%
\tilde{N}_{ji}$ of Chern-Cartan $(c.n.c.),$ with respect to the metric (\ref%
{II.1}). After a lot of trivial calculus, we obtain a simplified writing for
these%
\begin{equation}
\tilde{N}_{ji}=\overset{a}{N_{ji}}-\frac{1}{\gamma }\left( \frac{\partial
^{\ast }b_{\bar{r}}}{\partial z^{i}}\zeta ^{\bar{r}}-\frac{\alpha \beta }{%
|\beta |}\frac{\partial ^{\ast }b^{\bar{r}}}{\partial z^{i}}b_{\bar{r}%
}\right) \xi _{j}-\frac{\beta }{|\beta |}k_{j\bar{r}}\frac{\partial ^{\ast
}b^{\bar{r}}}{\partial z^{i}},  \label{II.8}
\end{equation}%
where $\xi _{j}:=\bar{\beta}\zeta _{j}+\alpha ^{2}b_{j},$ $k_{j\bar{r}%
}:=\alpha a_{j\bar{r}}+\frac{\alpha ||b||^{2}+|\beta |}{\gamma }\zeta
_{j}\zeta _{\bar{r}}-\frac{\alpha \mathcal{\tilde{C}}\beta }{\gamma |\beta |}%
b_{j}\zeta _{\bar{r}}$ and $\overset{a}{N_{ji}}:=-a_{j\bar{r}}\frac{\partial
^{\ast }a^{\bar{r}l}}{\partial z^{i}}\zeta _{l}$.

\begin{theorem}
Let $(M,\mathcal{\tilde{C}})$ be a connected non-purely Hermitian complex
Cartan-Randers space. Then, $(M,\mathcal{\tilde{C}})$ is a Berwald-Cartan
space if and only if $\overset{a}{\delta _{k}^{\ast }}|\beta |=0$, where $%
\overset{a}{\delta _{k}^{\ast }}$ is the adapted \ frame \ corresponding \
to $\overset{a}{N_{ji}}.$ Moreover, $\tilde{N}_{ji}=\overset{a}{N_{ji}}.$
\end{theorem}

\begin{proof}
If $(M,\mathcal{\tilde{C}})$ is Berwald then $N_{jk}=H_{jk}^{i}(z)\zeta _{i},
$ which means that $N_{jk}$ are homogeneous polynomials in $\zeta _{i}$ of
first degree. Thus, using (\ref{II.8}) we have

$\alpha |\beta |\{2(\tilde{N}_{ji}-\overset{a}{N_{ji}})$

$+\frac{\beta }{|\beta |^{2}}\frac{\partial ^{\ast }b^{\bar{r}}}{\partial
z^{i}}[a_{j\bar{r}}(\alpha ^{2}||b||^{2}+|\beta |^{2})+||b||^{2}\zeta
_{j}\zeta _{\bar{r}}-\alpha ^{2}b_{j}b_{\bar{r}}-(\bar{\beta}\zeta _{\bar{r}%
}b_{j}+\beta b_{\bar{r}}\zeta _{j})]\}$

$+(\alpha ^{2}||b||^{2}+|\beta |^{2})(\tilde{N}_{ji}-\overset{a}{N_{ji}})+%
\frac{\partial ^{\ast }b_{\bar{r}}}{\partial z^{i}}\zeta ^{\bar{r}}\xi
_{j}+\beta \frac{\partial ^{\ast }b^{\bar{r}}}{\partial z^{i}}(2\alpha
^{2}a_{j\bar{r}}+\zeta _{j}\zeta _{\bar{r}}-\frac{\alpha ^{2}\beta }{|\beta
|^{2}}\zeta _{\bar{r}}b_{j})=0,$ \newline
which contains an irrational part and a rational one. Thus, we obtain%
\begin{eqnarray*}
&&-\beta \frac{\partial ^{\ast }b^{\bar{r}}}{\partial z^{i}}[a_{j\bar{r}%
}(\alpha ^{2}||b||^{2}+|\beta |^{2})+||b||^{2}\zeta _{j}\zeta _{\bar{r}%
}-\alpha ^{2}b_{j}b_{\bar{r}}-(\beta \zeta _{\bar{r}}b_{j}+\bar{\beta}b_{%
\bar{r}}\zeta _{j})] \\
&=&2|\beta |^{2}(\tilde{N}_{ji}-\overset{a}{N_{ji}})\;\mbox{  and  }\; \\
&&-(\bar{\beta}\zeta ^{\bar{r}}\frac{\partial ^{\ast }b_{\bar{r}}}{\partial
z^{i}}+\beta \frac{\partial ^{\ast }b^{\bar{r}}}{\partial z^{i}}\zeta _{\bar{%
r}})\zeta _{j}-\frac{\alpha ^{2}\beta }{|\beta |^{2}}(\bar{\beta}\zeta ^{%
\bar{r}}\frac{\partial ^{\ast }b_{\bar{r}}}{\partial z^{i}}-\beta \frac{%
\partial ^{\ast }b^{\bar{r}}}{\partial z^{i}}\zeta _{\bar{r}})b_{j}-2\alpha
^{2}\beta \frac{\partial ^{\ast }b^{\bar{r}}}{\partial z^{i}}a_{j\bar{r}} \\
&=&(\alpha ^{2}||b||^{2}+|\beta |^{2})(\tilde{N}_{ji}-\overset{a}{N_{ji}}).
\end{eqnarray*}%
Contracting the above relations with $b^{j}$ and $\zeta ^{j},$ it results%
\begin{equation}
(\tilde{N}_{ji}-\overset{a}{N_{ji}})b^{j}=0;  \label{4.33;}
\end{equation}%
\begin{eqnarray}
2|\beta |^{2}(\tilde{N}_{ji}-\overset{a}{N_{ji}})\zeta ^{j}+2\alpha
^{2}\beta \frac{\partial ^{\ast }b^{\bar{r}}}{\partial z^{i}}(||b||^{2}\zeta
_{\bar{r}}-\bar{\beta}b_{\bar{r}}) &=&0;  \notag \\
(\alpha ^{2}||b||^{2}+|\beta |^{2})\zeta ^{\bar{r}}\frac{\partial ^{\ast }b_{%
\bar{r}}}{\partial z^{i}}-\frac{\beta ^{2}}{|\beta |^{2}}(\alpha
^{2}||b||^{2}-|\beta |^{2})\frac{\partial ^{\ast }b^{\bar{r}}}{\partial z^{i}%
}\zeta _{\bar{r}}+2\alpha ^{2}\beta \frac{\partial ^{\ast }b^{\bar{r}}}{%
\partial z^{i}}b_{\bar{r}} &=&0;  \notag \\
(\alpha ^{2}||b||^{2}+|\beta |^{2})(\tilde{N}_{ji}-\overset{a}{N_{ji}})\zeta
^{j}+2\alpha ^{2}(\bar{\beta}\zeta ^{\bar{r}}\frac{\partial ^{\ast }b_{\bar{r%
}}}{\partial z^{i}}+\beta \frac{\partial ^{\ast }b^{\bar{r}}}{\partial z^{i}}%
\zeta _{\bar{r}}) &=&0.  \notag
\end{eqnarray}

Adding the second and the third relations from (\ref{4.33;}), we obtain

$2|\beta |^{2}(\tilde{N}_{ji}-\overset{a}{N_{ji}})\zeta ^{j}+(\alpha
^{2}||b||^{2}+|\beta |^{2})(\bar{\beta}\zeta ^{\bar{r}}\frac{\partial ^{\ast
}b_{\bar{r}}}{\partial z^{i}}+\beta \frac{\partial ^{\ast }b^{\bar{r}}}{%
\partial z^{i}}\zeta _{\bar{r}})=0.$

This, together with the fourth equation from (\ref{4.33;}), implies $(\tilde{%
N}_{ji}-\overset{a}{N_{ji}})\zeta ^{j}=0$ and $\bar{\beta}\zeta ^{\bar{r}}%
\frac{\partial ^{\ast }b_{\bar{r}}}{\partial z^{i}}+\beta \frac{\partial
^{\ast }b^{\bar{r}}}{\partial z^{i}}\zeta _{\bar{r}}=0.$ The last condition
can be rewritten as $\overset{a}{\delta _{k}^{\ast }}|\beta |=0.$

Conversely, if $\overset{a}{\delta _{k}^{\ast }}|\beta |=0,$ by derivation
with respect to $\zeta _{l}$ and then with $\zeta _{\bar{m}}$, we deduce $b^{%
\bar{m}}a^{\bar{r}l}\frac{\partial ^{\ast }b_{\bar{r}}}{\partial z^{i}}+b^{l}%
\frac{\partial ^{\ast }b^{\bar{m}}}{\partial z^{i}}=0.$ The last relation
gives
\begin{equation*}
\frac{\partial ^{\ast }b^{\bar{m}}}{\partial z^{i}}=-\frac{\bar{\beta}}{%
|\beta |^{2}}b^{\bar{m}}\zeta ^{\bar{r}}\frac{\partial ^{\ast }b_{\bar{r}}}{%
\partial z^{i}},
\end{equation*}%
which substituted into (\ref{II.8}) implies $\tilde{N}_{ji}=\overset{a}{%
N_{ji}}$ and so, $\tilde{N}_{ji}$ are holomorphic in $\zeta _{l},$ i.e. the
space is Berwald.
\end{proof}

\begin{corollary}
Let $(M,\mathcal{\tilde{C}})$ be a connected non-purely Hermitian complex
Cartan-Randers space. Then, $(M,\mathcal{\tilde{C}})$ is a strongly
Berwald-Cartan space if and only if $\overset{a}{\delta _{k}^{\ast }}|\beta
|=0$ and $\alpha \,$\ is K\"{a}hler.
\end{corollary}

\begin{proof} It follows by Theorem 4.2. \end{proof}

\medskip In the remaining part of this section we come with some examples of
Cartan-Randers metrics which are Berwald-Cartan or strongly Berwald-Cartan.

\medskip \textbf{Example 4.1. }We consider $\Delta =\left\{
(z^{1},z^{2}):=(z,w)\in \mathbf{C}^{2},\;|w|<|z|<1\right\} $ and $(\zeta
_{1},\zeta _{2}):=(\zeta ,\upsilon )\in \widetilde{T^{\prime \ast }\Delta }.$
We choose $\beta =b^{i}(z,w)\zeta _{i},$ $i=1,2,$ with%
\begin{equation}
b^{1}=\frac{w}{|z|^{2}-|w|^{2}};\;b^{2}=-\frac{z}{|z|^{2}-|w|^{2}}.
\label{III.11}
\end{equation}

The function $\alpha :=\sqrt{a^{\bar{j}i}\zeta _{i}\zeta _{\bar{j}}}$, $%
i,j=1,2,$ where%
\begin{equation}
a^{\bar{1}1}=\frac{1}{\left( 1-|z|^{2}\right) ^{2}}+b^{1}b^{\bar{1}};\;a^{%
\bar{1}2}=b^{\bar{1}}b^{2};\;a^{\bar{2}2}=b^{\bar{2}}b^{2};  \label{iii.11'}
\end{equation}%
defines a purely Hermitian complex Cartan metric on $\Delta .$ Using these
tools, we obtain the complex Cartan-Randers metric%
\begin{equation}
\mathcal{\tilde{C}}=\alpha +|\beta |=\sqrt{\frac{|\zeta |^{2}}{\left(
1-|z|^{2}\right) ^{2}}+|\beta |^{2}}+|\beta |,  \label{iii.12}
\end{equation}%
with $||b||^{2}=1,$ $b_{1}=0$ and $b_{2}=-\frac{|z|^{2}-|w|^{2}}{z}.$ Some
computations lead us to the conclusion that the metric (\ref{iii.12}) is
Berwald-Cartan, that is,

$\overset{a}{\delta _{1}^{\ast }}|\beta |=\bar{\beta}\zeta ^{\bar{r}}\frac{%
\partial ^{\ast }b_{\bar{r}}}{\partial z}+\beta \frac{\partial ^{\ast }b^{%
\bar{r}}}{\partial z}\zeta _{\bar{r}}=\bar{\beta}\zeta ^{\bar{2}}\frac{%
\partial ^{\ast }b_{\bar{2}}}{\partial z}+\beta (\frac{\partial ^{\ast }b^{%
\bar{1}}}{\partial z}\bar{\zeta}+\frac{\partial ^{\ast }b^{\bar{2}}}{%
\partial z}\bar{\upsilon})=0$ and

$\overset{a}{\delta _{2}^{\ast }}|\beta |=\bar{\beta}\zeta ^{\bar{r}}\frac{%
\partial ^{\ast }b_{\bar{r}}}{\partial w}+\beta \frac{\partial ^{\ast }b^{%
\bar{r}}}{\partial w}\zeta _{\bar{r}}=\bar{\beta}\zeta ^{\bar{2}}\frac{%
\partial ^{\ast }b_{\bar{2}}}{\partial w}+\beta (\frac{\partial ^{\ast }b^{%
\bar{1}}}{\partial w}\bar{\zeta}+\frac{\partial ^{\ast }b^{\bar{2}}}{%
\partial w}\bar{\upsilon})=0.$

Note that the metric (\ref{iii.12}) is not one strongly Berwald-Cartan
because $\alpha $ from (\ref{iii.11'}) is not a K\"{a}hler metric.

\medskip \textbf{Example 4.2.} On $M=\mathbf{C}^{2}$ we set the purely
Hermitian metric%
\begin{equation}
\alpha ^{2}=e^{z^{1}+\bar{z}^{1}}\left\vert \zeta _{1}\right\vert
^{2}+e^{z^{2}+\bar{z}^{2}}\left\vert \zeta _{2}\right\vert ^{2}  \label{iii}
\end{equation}%
and we choose $\beta =e^{z^{2}}\zeta _{2}.$ Then, $\left\vert \beta
\right\vert ^{2}=e^{z^{2}+\bar{z}^{2}}\left\vert \zeta _{2}\right\vert ^{2}$
and so, $b_{1}=b^{1}=0,$ $b_{2}=e^{-z^{2}},$ $b^{2}=e^{z^{2}}$ and $||b||=1.$

With these tools we construct a complex Cartan-Randers metric%
\begin{equation}
\mathcal{\tilde{C}}=\sqrt{e^{z^{1}+\bar{z}^{1}}\left\vert \zeta
_{1}\right\vert ^{2}+e^{z^{2}+\bar{z}^{2}}\left\vert \zeta _{2}\right\vert
^{2}}+\sqrt{e^{z^{2}+\bar{z}^{2}}\left\vert \zeta _{2}\right\vert ^{2}},
\label{4.9}
\end{equation}%
which is non purely Hermitian, and $\det (\tilde{h}^{\bar{j}i})=\frac{%
\mathcal{\tilde{C}}^{4}}{2\alpha ^{3}|\beta |}\det (a^{\bar{j}i})>0,$ $%
i,j=1,2.$

The metric (\ref{4.9}) is Berwald-Cartan. Indeed,%
\begin{equation*}
\overset{a}{\delta _{i}^{\ast }}|\beta |=\bar{\beta}\zeta ^{\bar{2}}\frac{%
\partial ^{\ast }b_{\bar{2}}}{\partial z^{i}}+\beta \frac{\partial ^{\ast
}b^{\bar{2}}}{\partial z^{i}}\zeta _{\bar{2}}=0,\;i=1,2.
\end{equation*}%
Moreover, due to Theorem 4.2 we have%
\begin{equation*}
\tilde{N}_{11}=\overset{a}{N_{11}}=\zeta _{1}\;;\;\tilde{N}_{12}=\overset{a}{%
N_{12}}=\tilde{N}_{21}=\overset{a}{N_{21}}=0\;;\;\tilde{N}_{22}=\overset{a}{%
N_{22}}=\zeta _{2},
\end{equation*}%
which attest that the metric (\ref{iii}) is K\"{a}hler. Thus, by Corollary
4.1, (\ref{4.9}) is a strongly Berwald-Cartan metric. \medskip Note that the
above example can be generalized to a class of strongly Berwald-Cartan
metrics, taking on $M=\mathbf{C}^{n},$%
\begin{equation*}
\alpha ^{2}=\sum\limits_{k=1}^{n}e^{z^{k}+\bar{z}^{k}}\left\vert \zeta
_{k}\right\vert ^{2}.
\end{equation*}%
For $\beta $ we can choose for instance $\beta =e^{z^{k}}\eta ^{k},$ where $%
k=\overline{1,n}.$

\section{The $\mathcal{L}-$ duality between complex Finsler and complex
Cartan spaces}

\setcounter{equation}{0}Another way to describe the complex Cartan spaces is
given by the correspondence between the various geometrical objects on a
complex Finsler space $(M,F)$ and those of a complex Cartan space $(M,%
\mathcal{C})$, via the complex Legendre transformation (the $\mathcal{L}-$
dual process), \cite{Mub}.

\subsection{A nonholonomic vertical frame}

\setcounter{equation}{0}In our next approach we need another vertical frame
besides $\dot{\partial}^{k}=\frac{\partial }{\partial \zeta _{k}}$ and its
conjugate. Denoting $\zeta ^{k}:=\frac{\partial H}{\partial \zeta _{k}}=h^{%
\bar{j}k}\zeta _{\bar{j}},$ they change as $\zeta ^{\prime k}=\frac{\partial
z^{\prime k}}{\partial z^{l}}\zeta ^{l}.$ Corresponding to the new variables
$\zeta ^{k}$, further on we construct a frame, denoted by $\frac{\partial }{%
\partial \zeta ^{k}}$.

With respect to the vertical natural frame $\{\dot{\partial}^{k},\dot{%
\partial}^{\bar{k}}\}$ on $VT^{\prime \ast }M\oplus \overline{VT^{\prime
\ast }M},$ $\frac{\partial }{\partial \zeta ^{k}}$ can be decomposed as%
\begin{equation}
\frac{\partial }{\partial \zeta ^{k}}:=A_{km}\dot{\partial}^{m}+h_{k\bar{m}}%
\dot{\partial}^{\bar{m}},  \label{1}
\end{equation}%
where the tensor $A_{km}$ can be found. We require the conditions $\frac{%
\partial \zeta ^{j}}{\partial \zeta ^{k}}=\delta _{k}^{j}$ and $\frac{%
\partial \bar{\zeta}^{j}}{\partial \zeta ^{k}}=0$ and thus, it results%
\begin{equation}
A_{kr}h^{jr}=0\;;\;A_{kr}=-h_{r\bar{j}}h_{k\bar{m}}h^{\bar{j}\bar{m}}=A_{rk},
\label{2}
\end{equation}%
where $h^{jr}:=\frac{\partial ^{2}H}{\partial \zeta _{j}\partial \zeta _{r}}%
=(\dot{\partial}^{r}h^{\bar{p}j})\zeta _{\bar{p}}$ , which implies%
\begin{equation}
h_{r\bar{j}}h^{\bar{j}\bar{m}}h^{sr}=0\;;\;(\dot{\partial}^{j}A_{kr})\zeta
^{r}=0\;;\;(\dot{\partial}^{\bar{s}}A_{kr})\zeta ^{r}=h_{k\bar{m}}h^{\bar{m}%
\bar{s}}.  \label{3}
\end{equation}

\begin{lemma}
Let $(M,\mathcal{C})$ be a complex Cartan space. If the tensor $A_{km}$
satisfies the conditions (\ref{2}), then $A_{kr}=\frac{\partial ^{2}H}{%
\partial \zeta ^{k}\partial \zeta ^{r}}$ and $h_{k\bar{r}}=\frac{\partial
^{2}H}{\partial \zeta ^{k}\partial \bar{\zeta}^{r}}.$
\end{lemma}

\begin{proof}  Due to (\ref{1}) and (\ref{3}), we have

$\frac{\partial H}{\partial \zeta ^{k}}=A_{kr}(\dot{\partial}^{r}H)+h_{k\bar{%
m}}(\dot{\partial}^{\bar{m}}H)=A_{kr}\zeta ^{r}+\zeta _{k}$ and

$\frac{\partial ^{2}H}{\partial \zeta ^{k}\partial \zeta ^{r}}=\frac{%
\partial }{\partial \zeta ^{r}}\left( \frac{\partial H}{\partial \zeta ^{k}}%
\right) =(A_{rj}\dot{\partial}^{j}+h_{r\bar{s}}\dot{\partial}^{\bar{s}%
})\left( A_{kl}\zeta ^{l}+\zeta _{k}\right) $

$=A_{rj}(\dot{\partial}^{j}A_{kl})\zeta ^{l}+A_{rj}A_{kl}h^{lj}+A_{rk}+h_{r%
\bar{s}}(\dot{\partial}^{\bar{s}}A_{kl})\zeta ^{l}+h_{r\bar{s}}A_{kl}h^{\bar{%
s}l}$

$=2A_{rk}+A_{rj}(\dot{\partial}^{j}A_{kl})\zeta ^{l}+h_{r\bar{s}}(\dot{%
\partial}^{\bar{s}}A_{kl})\zeta ^{l}=A_{rk}.$

Also, it results $\frac{\partial ^{2}H}{\partial \zeta ^{k}\partial \bar{%
\zeta}^{r}}=\frac{\partial }{\partial \bar{\zeta}^{r}}\left( \frac{\partial H%
}{\partial \zeta ^{k}}\right) =(A_{\bar{r}\bar{j}}\dot{\partial}^{\bar{j}%
}+h_{s\bar{r}}\dot{\partial}^{s})(A_{kl}\zeta ^{l}+\zeta _{k})$

$=A_{\bar{r}\bar{j}}(\dot{\partial}^{\bar{j}}A_{kl})\zeta ^{l}+A_{\bar{r}%
\bar{j}}A_{kl}h^{\bar{j}l}+h_{s\bar{r}}(\dot{\partial}^{s}A_{kl})\zeta
^{l}+h_{r\bar{s}}A_{kl}h^{sk}+h_{k\bar{r}}$

$=A_{\bar{r}\bar{j}}h_{k\bar{m}}h^{\bar{m}\bar{j}}-A_{\bar{r}\bar{j}}h_{l%
\bar{s}}h_{k\bar{m}}h^{\bar{s}\bar{m}}h^{\bar{j}l}+h_{k\bar{r}}=h_{k\bar{r}},
$ which completes the proof.
\end{proof}

Now, we consider the Hessian matrix $\mathcal{H}_{1}=\left(
\begin{array}{cc}
h^{ks} & h^{\bar{p}k} \\
h^{\bar{r}s} & h^{\bar{p}\bar{r}}%
\end{array}%
\right) $ on $T_{C}(T^{\prime \ast }M)$, of the complex Cartan metric $%
H=H(z^{k},\zeta _{k})$. But, the complex Cartan metric $H$ can be seen as a
function of $(z^{k},\zeta ^{k})$ and so, its Hessian matrix on $%
T_{C}(T^{\prime \ast }M)$ is $\mathcal{H}_{2}=\left(
\begin{array}{cc}
h_{kj} & h_{j\bar{r}} \\
h_{k\bar{m}} & h_{\bar{m}\bar{r}}%
\end{array}%
\right) $, where $h_{kj}:=\frac{\partial ^{2}H}{\partial \zeta ^{k}\partial
\zeta ^{j}}$.

\begin{theorem}
Let $(M,\mathcal{C})$ be a complex Cartan space. Then, $\{\frac{\partial }{%
\partial \zeta ^{k}},\frac{\partial }{\partial \bar{\zeta}^{k}}\}$ is a
vertical frame on $VT^{\prime \ast }M\oplus \overline{VT^{\prime \ast }M},$
with%
\begin{equation}
\frac{\partial }{\partial \zeta ^{k}}:=h_{km}\dot{\partial}^{m}+h_{k\bar{m}}%
\dot{\partial}^{\bar{m}},  \label{2'}
\end{equation}%
if and only if $\mathcal{H}_{1}\mathcal{H}_{2}=\mathcal{H}_{2}\mathcal{H}%
_{1}=I_{2n}.$
\end{theorem}

\begin{proof}  This is immediate, taking into account that the condition $\mathcal{%
H}_{1}\mathcal{H}_{2}=\mathcal{H}_{2}\mathcal{H}_{1}=I_{2n}$ is equivalent
with (\ref{2}).
\end{proof}
Moreover, after some computation, we obtain the expression of $\dot{\partial}%
^{k}$, with respect to the frame $\{\frac{\partial }{\partial \zeta ^{k}},%
\frac{\partial }{\partial \bar{\zeta}^{k}}\},$%
\begin{equation}
\dot{\partial}^{k}=h^{kl}\frac{\partial }{\partial \zeta ^{l}}+h^{\bar{m}k}%
\frac{\partial }{\partial \bar{\zeta}^{m}},  \label{9}
\end{equation}%
under assumption $\mathcal{H}_{1}\mathcal{H}_{2}=\mathcal{H}_{2}\mathcal{H}%
_{1}=I_{2n}.$

In view of (\ref{2'}) this vertical frame is said to be \textit{nonholonomic}%
, because it depends on the tensors $h_{km}$ and $h_{k\bar{m}}.$

\begin{lemma}
Under assumptions (\ref{2}), we have that $\frac{\partial H}{\partial \zeta
^{k}}\zeta ^{k}=\frac{\partial H}{\partial \zeta ^{\bar{k}}}\zeta ^{\bar{k}%
}=H,$ $h_{rk}\zeta ^{r}=h_{rk}\zeta ^{k}=0$ and $h_{ij}$ is $(1,-1)$ -
homogeneous with respect to the variables $\zeta =(\zeta _{k}).\ $
\end{lemma}

\begin{proof} By (\ref{2'}) we have, $\frac{\partial H}{\partial \zeta ^{k}}\zeta ^{k}=h_{km}%
\frac{\partial H}{\partial \zeta _{m}}h^{\bar{p}k}\zeta _{\bar{p}}+h_{k\bar{m%
}}\frac{\partial H}{\partial \zeta _{\bar{m}}}h^{\bar{p}k}\zeta _{\bar{p}}.$
Due to (\ref{2}), the first term is vanishing and the second is $\frac{%
\partial H}{\partial \zeta _{\bar{m}}}\zeta _{\bar{m}}=H.$ So, $\frac{%
\partial H}{\partial \zeta ^{k}}\zeta ^{k}=H,$ and by conjugation, $\frac{%
\partial H}{\partial \zeta ^{\bar{k}}}\zeta ^{\bar{k}}=H.$

Now, $h_{rk}\zeta ^{r}=\frac{\partial ^{2}H}{\partial \zeta ^{k}\partial
\zeta ^{r}}\zeta ^{r}=\frac{\partial }{\partial \zeta ^{k}}(\frac{\partial H%
}{\partial \zeta ^{r}}\zeta ^{r})-\frac{\partial H}{\partial \zeta ^{k}}=0.$

Using again (\ref{2}), it results $(\dot{\partial}^{k}h_{ij})\zeta
_{k}=h_{ij}$ $\ $and $(\dot{\partial}^{\bar{k}}h_{ij})\zeta _{\bar{k}%
}=-h_{ij},$ which complete our claim.
\end{proof}

\subsection{$\mathcal{L}-$ dual process}

Let $(M,F)$ be a complex Finsler, where $F:T^{\prime }M\rightarrow \mathbb{R}%
^{+}$ is a continuous function satisfying the conditions:

\textit{i)} $L:=F^2$ is smooth on $\widetilde{T^{\prime }M}:=T^{\prime
}M\backslash \{0\};$

\textit{ii)} $F(z,\eta )\geq 0$, the equality holds if and only if $\eta =0;$

\textit{iii)} $F(z,\lambda \eta )=|\lambda |F(z,\eta )$ for $\forall \lambda
\in \mathbb{C}$;

\textit{iv)} the Hermitian matrix $\left( g_{i\bar{j}}(z,\eta )\right) $ is
positive definite, where $g_{i\bar{j}}:=\frac{\partial ^2L}{\partial \eta
^i\partial \bar{\eta}^j}$ is the fundamental metric tensor.

We consider the adapted frame $\{\delta _{k}:=\frac{\delta }{\delta z^{k}}=%
\frac{\partial }{\partial z^{k}}-N_{k}^{i}\frac{\partial }{\partial \eta ^{i}%
},\dot{\partial}_{k}:=\frac{\partial }{\partial \eta ^{k}}\}$ of
Chern-Finsler $(c.n.c.),$ where $N_{k}^{i}(z,\eta )=g^{\overline{m}i}\frac{%
\partial g_{l\overline{m}}}{\partial z^{k}}\eta ^{l},$ and the Chern-Finsler
connection (in brief $C-F$ connection), whose local coefficients are (see
\cite{Mub})
\begin{equation}
L_{jk}^{i}=g^{\overline{l}i}\delta _{k}g_{j\overline{l}}=\dot{\partial}%
_{j}N_{k}^{i}\;\;;\;C_{jk}^{i}=g^{\overline{l}i}\dot{\partial}_{k}g_{j%
\overline{l}},  \label{1.1}
\end{equation}%
and $L_{\bar{j}k}^{\bar{\imath}}=C_{\bar{j}k}^{\bar{\imath}}=0$. We recall
that in \cite{A-P}'s terminology, the complex Finsler space $(M,F)$ is
\textit{K\"{a}hler}$\;$iff $T_{jk}^{i}\eta ^{j}=0$ and \textit{weakly K\"{a}%
hler }iff\textit{\ } $g_{i\overline{l}}T_{jk}^{i}\eta ^{j}\overline{\eta }%
^{l}=0,$ where $T_{jk}^{i}:=L_{jk}^{i}-L_{kj}^{i}.$

The Chern-Finsler $(c.n.c.)$ does not generally derive from a spray, but it
always determines a complex spray with the local coefficients $G^{i}=\frac{1%
}{2}N_{j}^{i}\eta ^{j}.$

In \cite{Mub} the \textit{complex Legendre transformation} was introduced,
i.e. a local diffeomorphism $\Phi \times \bar{\Phi}$ with $\Phi :U\subset
T^{\prime }M\rightarrow \bar{U}^{\ast }\subset T^{\prime \prime \ast }M,$ $%
\Phi (z^{k},\eta ^{k})=(z^{k},\dot{\partial}_{\bar{k}}L),$ and $\bar{\Phi}:%
\bar{U}\subset T^{\prime \prime }M\rightarrow U^{\ast }\subset T^{\prime
\ast }M,$ $\bar{\Phi}(z^{k},\bar{\eta}^{k})=(z^{k},\dot{\partial}_{k}L).$
For simplicity, hereinafter the complex Legendre transformation is denoted
only by $\Phi $ and the distinction between the open sets $U$ and $\bar{U}$
is not specified, but we have assumed that it is defined, as above, on whole
$T_{C}M.$ The properties obtained by $\Phi $ or by $\Phi ^{-1}$are called $%
\mathcal{L}-$ dual one to another. Also, we can assume that in any point of $%
T_{C}M$ there are local charts which are sent by $\mathcal{L}-$ duality in
local charts on $T_{C}^{\ast }M$. We keep the notations from \cite{Mub},
(p.163), namely  '$^{\ast }$' is the image of various geometric objects by $%
\Phi $, and '$^{\circ }$' is their image by $\Phi ^{-1}.$

Now, setting the locally tangent map $d\Phi :T_{C}(T^{\prime }M)\rightarrow
T_{C}(T^{\prime \prime \ast }M)$ and $d\bar{\Phi}:T_{C}(T^{\prime \prime}
M)\rightarrow T_{C}(T^{\prime \ast }M),$ we determine the conditions under
which $d\Phi $ sends the complex tangent vectors in $T^{\prime }M$ into the
complex tangent vectors in $T^{\prime \ast }M,$ such that the image by
complex Legendre transformation, of a complex Finsler space $(M,F)$ is
locally a complex Cartan space $(M,\mathcal{C})$, and conversely, i.e.,%
\begin{equation}
\left( L(z^{k},\eta ^{k})\right) ^{\ast }=H(z^{k},\zeta _{k})\;;\;\left(
H(z^{k},\zeta _{k})\right) ^{\circ }=L(z^{k},\eta ^{k}),  \label{4.2}
\end{equation}%
with%
\begin{equation}
\frac{\partial L}{\partial z^{i}}=-\frac{\partial ^{\ast }H}{\partial z^{i}}%
\;;\;\left( \eta ^{k}\right) ^{\ast }=\dot{\partial}^{k}H\;;\;(\zeta
_{k})^{\circ }=\dot{\partial}_{k}L.  \label{4.0}
\end{equation}

\medskip Let $\mathcal{G}=\left(
\begin{array}{cc}
g_{kj} & g_{j\bar{r}} \\
g_{k\bar{m}} & g_{\bar{m}\bar{r}}%
\end{array}%
\right) $ be the Hessian matrix on $T_{C}(T^{\prime }M),$ of a complex
Finsler metric $L(z^{k},\eta ^{k})$, where $g_{jk}:=\frac{\partial ^{2}L}{%
\partial \eta ^{j}\partial \eta ^{k}}$ and $g_{j\bar{k}}:=\frac{\partial
^{2}L}{\partial \eta ^{j}\partial \bar{\eta }^{k}}$ is the metric tensor.

Since the Jacobi matrix of $\Phi $ is $\left(
\begin{array}{cc}
\delta _{jk} & \frac{\partial ^{2}L}{\partial z^{j}\partial \bar{\eta}^{r}}
\\
0 & g_{j\bar{r}}%
\end{array}%
\right) ,$ then $d\Phi $ sends $\{\frac{\partial }{\partial z^{k}},\dot{%
\partial}_{k}\}$ in $\{(\frac{\partial }{\partial z^{k}})^{\ast },(\dot{%
\partial}_{k})^{\ast }\}$ by%
\begin{eqnarray}
d\Phi (\frac{\partial }{\partial z^{k}}) &=&(\frac{\partial }{\partial z^{k}}%
)^{\ast }=\frac{\partial ^{\ast }}{\partial z^{k}}+\frac{\partial ^{2}L}{%
\partial z^{k}\partial \eta ^{j}}\dot{\partial}^{j}+\frac{\partial ^{2}L}{%
\partial z^{k}\partial \bar{\eta}^{r}}\dot{\partial}^{\bar{r}};
\label{4.2''} \\
d\Phi (\dot{\partial}_{k}) &=&(\dot{\partial}_{k})^{\ast }=g_{kj}\dot{%
\partial}^{j}+g_{k\bar{r}}\dot{\partial}^{\bar{r}},  \notag
\end{eqnarray}%
with conditions%
\begin{equation}
(\dot{\partial}_{i}\eta ^{j})^{\ast }=(\dot{\partial}_{i})^{\ast }(\eta
^{j})^{\ast }=\delta _{i}^{j}\;;\;(\dot{\partial}_{i}\bar{\eta}^{j})^{\ast
}=(\dot{\partial}_{i})^{\ast }(\bar{\eta}^{j})^{\ast }=0.  \label{b}
\end{equation}%
which by $\left( \eta ^{k}\right) ^{\ast }=\dot{\partial}^{k}H,$ lead to%
\begin{equation}
g_{k\bar{r}}h^{\bar{r}j}+g_{ik}h^{sk}=\delta _{k}^{j}\;;\;g_{ik}h^{\bar{m}%
k}+g_{i\bar{k}}h^{\bar{s}\bar{k}}=0,  \notag
\end{equation}%
equivalently with $\mathcal{GH}_{1}=\mathcal{H}_{1}\mathcal{G}=I_{2n}.$

Similarly, we obtain that $d\Phi ^{-1}$ sends $\{\frac{\partial ^{\ast }}{%
\partial z^{k}},\dot{\partial}^{k}\}$ in $\{(\frac{\partial ^{\ast }}{%
\partial z^{k}})^{\circ },(\dot{\partial}^{k})^{\circ }\}$ if and only if $%
\mathcal{GH}_{1}=\mathcal{H}_{1}\mathcal{G}=I_{2n}.$ So, we have proved.

\begin{theorem}
Let $M$ $\ $be a complex manifold with the metrics $F$ and $\mathcal{C}$
given by (\ref{4.2}). The $\mathcal{L}-$ dual of the complex Finsler space $%
(M,F)$ is locally the complex Cartan space $(M,\mathcal{C})$ if and only if $%
\mathcal{GH}_{1}=\mathcal{H}_{1}\mathcal{G}=I_{2n}.$
\end{theorem}

An immediate consequence of the above Theorem is that the Hessian matrix $%
\mathcal{G}$ is invertible. Of course, hereinafter all considerations on $%
\mathcal{L}-$ dual process can be made under assumption that the complex
Finsler space $(M,F)$ has locally $\det \mathcal{G}\neq 0.$ A such complex
Finsler space will be called \textit{locally regulate}. Looking back on
Theorem 5.1, we conclude that in the $\mathcal{L}-$ dual Cartan space, of a
locally regulate Finsler space, there exists a nonholonomic vertical frame
in any local chart.

Now, replacing (\ref{9}) in the second relation from (\ref{4.2''}), we find
that the image by $\Phi $ of the frame $\dot{\partial}_{k}$ is the vertical
frame (\ref{9}), i.e. $(\dot{\partial}_{k})^{\ast }=\frac{\partial }{%
\partial \zeta ^{k}},$ which together (\ref{4.2}), yields%
\begin{eqnarray}
(g_{kj})^{\ast } &=&h_{kj}\;;\;(g_{j\bar{r}})^{\ast }=h_{j\bar{r}%
}\;;\;(g^{ks})^{\ast }=h^{ks}\;;\;(g^{\bar{p}k})^{\ast }=h^{\bar{p}k};
\label{10} \\
(\dot{\partial}_{k})^{\ast }(\frac{\partial L}{\partial z^{i}}) &=&-\frac{%
\partial }{\partial \zeta ^{k}}(\frac{\partial ^{\ast }H}{\partial z^{i}}%
)=-h_{kr}\dot{\partial}^{r}(\frac{\partial ^{\ast }H}{\partial z^{i}})-h_{k%
\bar{m}}\dot{\partial}^{\bar{m}}(\frac{\partial ^{\ast }H}{\partial z^{i}});
\notag \\
(\dot{\partial}_{\bar{k}})^{\ast }(\frac{\partial L}{\partial z^{i}}) &=&-%
\frac{\partial }{\partial \bar{\zeta}^{k}}(\frac{\partial ^{\ast }H}{%
\partial z^{i}})=-h_{\bar{r}\bar{k}}\dot{\partial}^{\bar{r}}(\frac{\partial
^{\ast }H}{\partial z^{i}})-h_{l\bar{k}}\dot{\partial}^{l}(\frac{\partial
^{\ast }H}{\partial z^{i}}).  \notag
\end{eqnarray}

Moreover, using (\ref{4.2''}) and $C-F$ $(c.n.c.),$ $N_{i}^{k}=g^{\overline{m%
}k}\frac{\partial g_{l\overline{m}}}{\partial z^{i}}\eta ^{l}=g^{\overline{m}%
k}\frac{\partial ^{2}L}{\partial z^{i}\partial \bar{\eta}^{m}}$, it results%
\begin{equation*}
(\frac{\partial }{\partial z^{i}})^{\ast }-N_{i}^{k}(\dot{\partial}%
_{k})^{\ast }=\frac{\partial ^{\ast }}{\partial z^{i}}+\mathring{N}_{ki}\dot{%
\partial}^{k},
\end{equation*}%
with $\mathring{N}_{ki}:=\frac{\partial ^{2}L}{\partial z^{i}\partial \eta
^{k}}-g_{jk}N_{i}^{j}.$ Next, due to (\ref{10}), we obtain that the image by
$\mathcal{L}-$ duality of $\mathring{N}_{ki}$ is $C-C$ $(c.n.c.),$ i.e. $(%
\mathring{N}_{ki})^{\ast }=N_{ki}$ and so, $(\delta _{k})^{\ast }=\delta
_{k}^{\ast }.$

\begin{proposition}
The $\mathcal{L}-$ dual of the $C-F$ connection is a connection $D\Gamma
^{\ast }$ with local coefficients%
\begin{eqnarray*}
H_{jk}^{\ast i} &=&H_{jk}^{i}\;;\;H_{j\bar{k}}^{\ast i}=0; \\
V_{j}^{\ast ik} &=&-h^{kr}h_{r\bar{m}}h_{j\bar{s}}(\dot{\partial}^{\bar{m}%
}h^{\bar{s}i})\;;\;V_{j}^{\ast i\bar{k}}=h^{\bar{k}r}h_{rl}V_{j}^{il}-h_{j%
\bar{s}}(\dot{\partial}^{\bar{k}}h^{\bar{s}i}).
\end{eqnarray*}
\end{proposition}

\begin{proof} Indeed,

$H_{jk}^{\ast i}:=(L_{jk}^{i})^{\ast }=[g^{\overline{l}i}(\delta _{k}g_{j%
\overline{l}})]^{\ast }=h^{\overline{l}i}(\delta _{k}^{\ast }h_{j\overline{l}%
})=H_{jk}^{i}$.
\medskip
Now,

$(C_{jk}^{i})^{\ast }=[g^{\overline{l}i}(\dot{\partial}_{k}g_{j\overline{l}%
})]^{\ast }=h^{\overline{l}i}\frac{\partial h_{j\overline{l}}}{\partial
\zeta ^{k}}=h^{\overline{l}i}h_{kr}(\dot{\partial}^{r}h_{j\overline{l}})+h^{%
\overline{l}i}h_{k\bar{m}}(\dot{\partial}^{\bar{m}}h_{j\overline{l}})$

$=-h_{kr}h_{j\bar{s}}(\dot{\partial}^{r}h^{\bar{s}i})-h_{k\bar{m}}h_{j\bar{s}%
}(\dot{\partial}^{\bar{m}}h^{\bar{s}i}).$ Thus,

$V_{j}^{\ast ik}:=h^{kr}(C_{jr}^{i})^{\ast }=-h^{kr}h_{r\bar{m}}h_{j\bar{s}}(%
\dot{\partial}^{\bar{m}}h^{\bar{s}i})$ and

$V_{j}^{\ast i\bar{k}}:=h^{\bar{k}r}(C_{jr}^{i})^{\ast }=h^{\bar{k}%
r}h_{rl}V_{j}^{il}-h_{j\bar{s}}(\dot{\partial}^{\bar{k}}h^{\bar{s}i}).$
\end{proof}

\begin{proposition}
Let $M$ $\ $be a complex manifold with the $\mathcal{L}-$ dual metrics $F$
and $\mathcal{C}$ given by (\ref{4.2}). If $F$ is a weakly K\"{a}hler
metric, then $\mathcal{C}$ is a weakly K\"{a}hler metric, too.
\end{proposition}

\begin{proof} The weakly K\"{a}hler property of $F$ can be rewritten as

$0=g_{i\bar{l}}(L_{jk}^{i}-L_{kj}^{i})\eta ^{j}\bar{\eta}^{l}=\frac{\partial
L}{\partial z^{k}}-\mathring{N}_{kj}\eta ^{j}=-(\frac{\partial ^{\ast }H}{%
\partial z^{k}}+\mathring{N}_{jk}\eta ^{j})+(\mathring{N}_{jk}-\mathring{N}%
_{kj})\eta ^{j}.$

Due to (\ref{4.2}) and $(\mathring{N}_{ki})^{\ast }=N_{ki}$, the image by $%
\mathcal{L}-$ duality of the last relation is

$0=-\frac{\delta ^{\ast }H}{\delta z^{k}}+(N_{jk}-N_{kj})\zeta ^{j}.$

Since $N_{ki}$ are local coefficients of $C-C$ $(c.n.c.),$ then $\frac{%
\delta ^{\ast }H}{\delta z^{k}}=0$ and so, $(N_{jk}-N_{kj})\zeta ^{j}=0,$
which gives our claim.
\end{proof}

\subsection{Geodesic curves of a complex Cartan space}

Let $\sigma :[a,b]\rightarrow \widetilde{T^{\prime }M}$ be a parametrized
curve which, in a local chart of $\widetilde{T^{\prime }M}$, $\ $is given by
$\sigma (s)=\left( z^{k}(s),\eta ^{k}(s)\right) ,$ $s\in \lbrack a,b],$ $k=%
\overline{1,n}$, where $\eta ^{k}(s)=\frac{dz^{k}}{ds}$ is a tangent vector
to the curve $(z^{k}(s))$ on $M.$

In \cite{A-P}' s sense, $\sigma (s)$ is a complex geodesic curve on a
complex Finsler space $(M,F)$ if and only if it is solution of the equations%
\begin{equation}
\frac{d^{2}z^{k}}{ds^{2}}+2G^{k}(z(s),\frac{dz}{ds})=\Theta ^{k}(z(s),\frac{%
dz}{ds})\;;\;k=\overline{1,n},  \label{2.1}
\end{equation}%
where $\Theta ^{k}=2g^{\bar{j}k}(\overset{c}{\delta _{\bar{j}}}L)=g^{\bar{m}%
k}g_{j\bar{l}}(L_{\bar{n}\bar{m}}^{\bar{l}}-L_{\bar{m}\bar{n}}^{\bar{l}%
})\eta ^{j}\bar{\eta}^{n}.$ Note that $\Theta ^{k}$ is vanishing if and only
if the space $(M,F)$ is weakly K\"{a}hler.

With notation $T^{k}:=\frac{d}{ds}(\frac{dz^{k}}{ds})+N_{l}^{k}\frac{dz^{l}}{%
ds},$ the equations (\ref{2.1}) become%
\begin{equation}
T^{k}=\Theta ^{k}\;;\;k=\overline{1,n}.  \label{2.2}
\end{equation}

Let us consider $(M,F)$ a locally regulate complex Finsler space and $(M,%
\mathcal{C})$ is the complex Cartan space, obtained by $\mathcal{L}-$ dual
process. Further on, our goal is to determine the image by $\mathcal{L}-$
duality of the complex geodesic curve $\sigma (s).$ The first question is
how to map a curve from $T^{\prime }M$ into a curve on $T^{\prime \ast }M$
or, more precisely, on $T^{\prime \prime \ast }M$?

\medskip By definition, a curve on $T^{\prime \ast }M$ is a map%
\begin{equation*}
s\rightarrow \sigma ^{\ast }(s)=(z^{k}(s),\zeta _{k}(s)),
\end{equation*}
where $\zeta _{k}(s)$ are the components of $(1,0)$-form. We know the
isomorphism between the tangent and cotangent spaces, via a metric tensor,
but this is defined on $M.$

On a complex Cartan space $(M,\mathcal{C}),$ with the metric tensor $h^{\bar{%
j}k}(z,\zeta ),$ we can consider the tangent vector (from $T_{C,u^{\ast
}}(T^{\prime \ast }M)$) to a curve $\sigma _{1}^{\ast }(s),$ given by $X=(%
\frac{dz^{k}}{ds},\zeta ^{k}=h^{\bar{m}k}\zeta _{\bar{m}}(s)).$ When $\zeta
^{k}=\frac{dz^{k}}{ds}$, that is $\sigma _{1}^{\ast }(s)=\sigma ^{\ast
}(s)=(z^{k}(s),\zeta _{k}=h_{k\bar{m}}\frac{d\bar{z}^{m}}{ds})$, we say that
$\sigma ^{\ast }(s)$ is the image by $\mathcal{L}-$ duality of the curve $%
\sigma (s)=(z^{k}(s),\eta ^{k}(s))$ from $T^{\prime }M,$ where $\eta ^{k}(s)=%
\frac{dz^{k}}{ds}.$ It is clear that $\zeta ^{k}=\frac{dz^{k}}{ds}$\ are the
components of a tangent vector to the curve $\sigma ^{\ast }(s),$ which is $%
\mathcal{L}-$ dual of $\sigma (s).$ Making an excessive use of notation, we
write $\sigma ^{\ast }(s):=[\sigma (s)]^{\ast },$ which in a local chart of $%
\widetilde{T^{\prime \ast }M}$, $\ $is $z^{k}=z^{k}(s),$ $\zeta _{k}=\zeta
_{k}(s),$ $k=\overline{1,n}.$

For two $\mathcal{L}-$ dual curves, we have%
\begin{equation}
\left( T^{k}\right) ^{\ast }=\left( \Theta ^{k}\right) ^{\ast }\;;\;k=%
\overline{1,n},  \label{2.3}
\end{equation}%
with $\frac{dz^{k}}{ds}=[\eta ^{k}(s)]^{\ast }=\zeta ^{k}(s).$

Taking into account (\ref{4.2}) and Proposition 5.2, we obtain%
\begin{equation}
\Theta ^{\ast k}:=\left( \Theta ^{k}\right) ^{\ast }=h^{\bar{m}k}(N_{\bar{r}%
\bar{m}}-N_{\bar{m}\bar{r}})\zeta ^{\bar{r}},  \label{2.4}
\end{equation}%
which is $(1,1)$-homogeneous with respect to the variables $\zeta =(\zeta
_{k}).\ $

Moreover, the space $(M,\mathcal{C})$ is weakly K\"{a}hler iff $\Theta
^{\ast k}=0.$

\medskip In order to obtain $\left( T^{k}\right) ^{\ast }$, we can rewrite $%
\frac{d^{2}z^{k}}{ds^{2}}$ as follows:

\medskip $\frac{d}{ds}(\frac{dz^{k}}{ds})=\frac{d}{ds}\eta ^{k}=\frac{d}{ds}%
(g^{\bar{m}k}\eta _{\bar{m}})=$

$=[\frac{\partial g^{\bar{m}k}}{\partial z^{l}}\frac{dz^{l}}{ds}+\frac{%
\partial g^{\bar{m}k}}{\partial \bar{z}^{r}}\frac{d\bar{z}^{r}}{ds}+(\,\dot{%
\partial}_{l}g^{\bar{m}k})\frac{d\eta ^{l}}{ds}+(\,\dot{\partial}_{\bar{r}%
}g^{\bar{m}k})\frac{d\bar{\eta}^{r}}{ds}]\eta _{\bar{m}}+g^{\bar{m}k}\frac{%
d\eta _{\bar{m}}}{ds}$

$=-g^{\bar{m}i}g^{\bar{n}k}[\frac{\partial g_{i\bar{n}}}{\partial z^{l}}%
\frac{dz^{l}}{ds}+\frac{\partial g_{i\bar{n}}}{\partial \bar{z}^{r}}\frac{d%
\bar{z}^{r}}{ds}+(\,\dot{\partial}_{l}g_{i\bar{n}})\frac{d\eta ^{l}}{ds}+(\,%
\dot{\partial}_{\bar{r}}g_{i\bar{n}})\frac{d\bar{\eta}^{r}}{ds}]\eta _{\bar{m%
}}+g^{\bar{m}k}\frac{d\eta _{\bar{m}}}{ds}$

$=-g^{\bar{n}k}[\frac{\partial g_{i\bar{n}}}{\partial z^{l}}\frac{dz^{l}}{ds}%
+\frac{\partial g_{i\bar{n}}}{\partial \bar{z}^{r}}\frac{d\bar{z}^{r}}{ds}%
+(\,\dot{\partial}_{l}g_{i\bar{n}})\frac{d\eta ^{l}}{ds}+(\,\dot{\partial}_{%
\bar{r}}g_{i\bar{n}})\frac{d\bar{\eta}^{r}}{ds}]\eta ^{i}+g^{\bar{m}k}\frac{%
d\eta _{\bar{m}}}{ds}$

$=-N_{l}^{k}\frac{dz^{l}}{ds}-g^{\bar{n}k}\frac{\partial ^{2}L}{\partial
\bar{z}^{r}\partial \bar{\eta}^{n}}\frac{d\bar{z}^{r}}{ds}-g^{\bar{n}k}\frac{%
\partial ^{2}L}{\partial \bar{\eta}^{r}\partial \bar{\eta}^{n}}\frac{d\bar{%
\eta}^{r}}{ds}+g^{\bar{m}k}\frac{d\eta _{\bar{m}}}{ds}.$

This implies,

$\frac{d}{ds}(\frac{dz^{k}}{ds})+N_{l}^{k}\frac{dz^{l}}{ds}=-g^{\bar{n}k}(%
\frac{\partial ^{2}L}{\partial \bar{z}^{r}\partial \bar{\eta}^{n}}\frac{d%
\bar{z}^{r}}{ds}+g_{\bar{r}\bar{n}}\frac{d\bar{\eta}^{r}}{ds}-\frac{d\eta _{%
\bar{n}}}{ds})$

$=-g^{\bar{n}k}[(\frac{\partial ^{2}L}{\partial \bar{z}^{r}\partial \bar{\eta%
}^{n}}-g_{\bar{s}\bar{n}}N_{\bar{r}}^{\bar{s}})\frac{d\bar{z}^{r}}{ds}+g_{%
\bar{s}\bar{n}}N_{\bar{r}}^{\bar{s}}\frac{d\bar{z}^{r}}{ds}+g_{\bar{r}\bar{n}%
}\frac{d^{2}\bar{z}^{r}}{ds^{2}}-\frac{d\eta _{\bar{n}}}{ds})$

$=-g^{\bar{n}k}[\mathring{N}_{\bar{n}\bar{r}}\frac{d\bar{z}^{r}}{ds}+g_{\bar{%
r}\bar{n}}(\frac{d^{2}\bar{z}^{r}}{ds^{2}}+N_{\bar{m}}^{\bar{r}}\frac{d\bar{z%
}^{m}}{ds})-\frac{d\eta _{\bar{n}}}{ds}],$ which gives%
\begin{equation}
T^{k}=-g^{\bar{n}k}[\mathring{N}_{\bar{n}\bar{r}}\frac{d\bar{z}^{r}}{ds}+g_{%
\bar{r}\bar{n}}T^{\bar{r}}-\frac{d\eta _{\bar{n}}}{ds}],  \label{2.5}
\end{equation}%
and by $\mathcal{L}-$ duality, it leads to%
\begin{equation}
\left( T^{k}\right) ^{\ast }=-h^{\bar{n}k}[N_{\bar{n}\bar{r}}\frac{d\bar{z}%
^{r}}{ds}+h_{\bar{r}\bar{n}}(T^{\bar{r}})^{\ast }-\frac{d\zeta _{\bar{n}}}{ds%
}].  \label{2.6}
\end{equation}

Now, substituting (\ref{2.3}) and (\ref{2.4}) into (\ref{2.6}), it results%
\begin{equation}
\Theta ^{\ast k}=-h^{\bar{n}k}[N_{\bar{n}\bar{r}}\frac{d\bar{z}^{r}}{ds}+h_{%
\bar{r}\bar{n}}\Theta ^{\ast \bar{r}}-\frac{d\zeta _{\bar{n}}}{ds}].
\label{2.7}
\end{equation}%
This is equivalent with

$\Theta ^{\ast k}+h^{\bar{n}k}h_{\bar{r}\bar{n}}\Theta ^{\ast \bar{r}}=h^{%
\bar{n}k}[\frac{d\zeta _{\bar{n}}}{ds}-N_{\bar{r}\bar{n}}\frac{d\bar{z}^{r}}{%
ds}]+h^{\bar{n}k}(N_{\bar{r}\bar{n}}-N_{\bar{n}\bar{r}})\frac{d\bar{z}^{r}}{%
ds}$

$=h^{\bar{n}k}[\frac{d\zeta _{\bar{n}}}{ds}-N_{\bar{r}\bar{n}}\frac{d\bar{z}%
^{r}}{ds}]+\Theta ^{\ast k}$ which leads to $\frac{d\zeta _{\bar{n}}}{ds}-N_{%
\bar{r}\bar{n}}\frac{d\bar{z}^{r}}{ds}=h_{\bar{r}\bar{n}}\Theta ^{\ast \bar{r%
}}.$ Thus, we have proved.

\begin{theorem}
Let $M$ $\ $be a complex manifold with the $\mathcal{L}-$ dual metrics $F$
and $\mathcal{C}$ given by (\ref{4.2}). If $\sigma (s)$ is a complex
geodesic curve on the complex Finsler space $(M,F)$ then its image by $%
\mathcal{L}-$ duality $\sigma ^{\ast }(s)$ satisfies the equations%
\begin{equation}
\frac{dz^{k}}{ds}=\zeta ^{k}\;;\;\;\frac{d\zeta _{k}}{ds}-N_{jk}\frac{dz^{j}%
}{ds}=h_{jk}\Theta ^{\ast j}\;;\;\;k=\overline{1,n}.  \label{2.8}
\end{equation}
Moreover, if $(M,F)$ is weakly K\"{a}hler, then $\sigma ^{\ast }(s)$
satisfies%
\begin{equation}
\frac{dz^{k}}{ds}=\zeta ^{k}\;;\;\;\frac{d\zeta _{k}}{ds}-N_{jk}\frac{dz^{j}%
}{ds}=0\;;\;\;k=\overline{1,n}.  \label{2.9}
\end{equation}
\end{theorem}

It is natural to ask if $\sigma ^{\ast }(s)$ is a geodesic curve for $%
\mathcal{C}$. The answer is provided below. We start with the fact that the
image by $\mathcal{L}-$ duality of the Euler-Lagrange equations for $\sigma
(s)$ are the Hamilton-Jacobi equations for $\sigma ^{\ast }(s),$%
\begin{equation}
\frac{dz^{k}}{ds}=\frac{\partial H}{\partial \zeta _{k}}\;;\;\frac{d\zeta
_{k}}{ds}=-\frac{\partial H}{\partial z^{k}}\;;\;k=\overline{1,n}.
\label{2.10}
\end{equation}

But the equations (\ref{2.9}) are equivalent with (\ref{2.10}). Indeed, the
first equation from (\ref{2.10}) is $\frac{dz^{k}}{ds}=h^{\bar{m}k}\zeta _{%
\bar{m}}={\zeta}^k $. The second equation (\ref{2.10}) can be rewritten as $%
\frac{d\zeta _{k}}{ds}=-\frac{\partial }{\partial z^{k}}\left( h^{\bar{m}%
l}\zeta _{l}\zeta _{\bar{m}}\right) .$ This is equivalent with $\frac{d\zeta
_{k}}{ds}=-\frac{\partial h^{\bar{m}l}}{\partial z^{k}}\zeta _{l}\zeta _{%
\bar{m}},$ which leads to $\frac{d\zeta _{k}}{ds}-N_{jk}\zeta ^{j}=0.$ So,
under assumption of weakly K\"{a}hler for the metric $\mathcal{C}$, the
curve $\sigma ^{\ast }(s)$ is solution of the Hamilton-Jacobi equations (\ref%
{2.10}).

Since in the weakly K\"{a}hler case for $F$, the Euler-Lagrange equations
give the geodesics for $F$, (see \cite{Mub}, p.101), we call the curve $%
\sigma ^{\ast }(s),$ which satisfies the equations (\ref{2.8}), \textit{the
complex geodesics curve for }$\mathcal{C}$. Note that when we say complex
geodesic curve for\ $\mathcal{C}$, we simply mean the curves which are the
image by $\mathcal{L}-$ duality of a complex geodesic curve on $(M,F)$.

\begin{theorem}
Let $(M,\mathcal{C})$ $\ $be a complex Cartan space. Then $\sigma ^{\ast
}(s) $ is a complex geodesic curve for $\mathcal{C}$ if and only if%
\begin{equation}
\frac{dz^{k}}{ds}=\zeta ^{k}\;;\;\frac{d^{2}z^{k}}{ds^{2}}+H_{jl}^{k}\zeta
^{j}\zeta ^{l}=\Theta ^{\ast k}\;;\;k=\overline{1,n}.  \label{2.11}
\end{equation}
\end{theorem}

\begin{proof}  We suppose that $\sigma ^{\ast }(s)$ is a complex geodesic for $%
\mathcal{C}$, i.e., it satisfies (\ref{2.8}). Differentiating the equation $%
\frac{dz^{k}}{ds}=\zeta ^{k}$ with respect to $s,$ it results

$\frac{d^{2}z^{k}}{ds^{2}}=\frac{d\zeta ^{k}}{ds}=\frac{d(h^{\bar{m}k}\zeta
_{\bar{m}})}{ds}$

$=[\frac{\partial h^{\bar{m}k}}{\partial z^{l}}\frac{dz^{l}}{ds}+\frac{%
\partial h^{\bar{m}k}}{\partial \bar{z}^{r}}\frac{d\bar{z}^{r}}{ds}+(\,\dot{%
\partial}^{l}h^{\bar{m}k})\frac{d\zeta _{l}}{ds}+(\,\dot{\partial}^{\bar{r}%
}h^{\bar{m}k})\frac{d\zeta _{\bar{r}}}{ds}]\zeta _{\bar{m}}+h^{\bar{m}k}%
\frac{d\zeta _{\bar{m}}}{ds}$

$=[\frac{\partial h^{\bar{m}k}}{\partial z^{l}}\frac{dz^{l}}{ds}+\frac{%
\partial h^{\bar{m}k}}{\partial \bar{z}^{r}}\frac{d\bar{z}^{r}}{ds}+(\,\dot{%
\partial}^{l}h^{\bar{m}k})\frac{d\zeta _{l}}{ds}]\zeta _{\bar{m}}+h^{\bar{m}%
k}\frac{d\zeta _{\bar{m}}}{ds}$

$=[\frac{\partial h^{\bar{m}k}}{\partial z^{l}}\zeta ^{l}+(\,\dot{\partial}%
^{l}h^{\bar{m}k})\frac{d\zeta _{l}}{ds}]\zeta _{\bar{m}}+h^{\bar{m}k}(\frac{%
d\zeta _{\bar{m}}}{ds}-N_{\bar{m}\bar{r}}\frac{d\bar{z}^{r}}{ds})$

$=[\frac{\partial h^{\bar{m}k}}{\partial z^{l}}\zeta ^{l}+(\,\dot{\partial}%
^{l}h^{\bar{m}k})\frac{d\zeta _{l}}{ds}]\zeta _{\bar{m}}+h^{\bar{m}k}(\frac{%
d\zeta _{\bar{m}}}{ds}-N_{\bar{r}\bar{m}}\frac{d\bar{z}^{r}}{ds})$

$+h^{\bar{m}k}(N_{\bar{r}\bar{m}}-N_{\bar{m}\bar{r}})\frac{d\bar{z}^{r}}{ds}$

$=[\frac{\partial h^{\bar{m}k}}{\partial z^{l}}\zeta ^{l}+N_{jl}(\,\dot{%
\partial}^{j}h^{\bar{m}k})\zeta ^{l}-N_{jl}(\,\dot{\partial}^{j}h^{\bar{m}%
k})\zeta ^{l}+(\,\dot{\partial}^{l}h^{\bar{m}k})\frac{d\zeta _{l}}{ds}]\zeta
_{\bar{m}}$

$+h^{\bar{m}k}h_{\bar{r}\bar{m}}\Theta ^{\ast \bar{r}}+\Theta ^{\ast k}$

$=(\delta _{l}^{\ast }h^{\bar{m}k})\zeta ^{l}\zeta _{\bar{m}}+(\,\dot{%
\partial}^{l}h^{\bar{m}k})(\frac{d\zeta _{l}}{ds}-N_{lj}\zeta ^{j})\zeta _{%
\bar{m}}+h^{\bar{m}k}h_{\bar{r}\bar{m}}\Theta ^{\ast \bar{r}}+\Theta ^{\ast
k}$

$=-h^{\bar{m}p}h^{\bar{n}k}(\delta _{l}^{\ast }h_{p\bar{n}})\zeta ^{l}\zeta
_{\bar{m}}+(\,\dot{\partial}^{l}h^{\bar{m}k})(\frac{d\zeta _{l}}{ds}%
-N_{jl}\zeta ^{j})\zeta _{\bar{m}}$

$+(\,\dot{\partial}^{l}h^{\bar{m}k})(N_{jl}-N_{lj})\zeta ^{j}\zeta _{\bar{m}%
}+h^{\bar{m}k}h_{\bar{r}\bar{m}}\Theta ^{\ast \bar{r}}+\Theta ^{\ast k}$

$=-H_{pl}^{k}\zeta ^{p}\zeta ^{l}+h^{kl}h_{jl}\Theta ^{\ast j}+(h^{kl}h_{l%
\bar{r}}+h^{\bar{m}k}h_{\bar{r}\bar{m}})\Theta ^{\ast \bar{r}}+\Theta ^{\ast
k},$ which due to (\ref{2}) gives (\ref{2.11})

Conversely, we suppose that $\sigma ^{\ast }(s)$ is solution of (\ref{2.11}%
). As above calculus, we obtain%
\begin{eqnarray}
&&\frac{d^{2}z^{k}}{ds^{2}}+H_{pl}^{k}\zeta ^{p}\zeta ^{l}=h^{kl}(\frac{%
d\zeta _{l}}{ds}-N_{jl}\frac{dz^{j}}{ds})+h^{kl}h_{l\bar{r}}\Theta ^{\ast
\bar{r}}  \label{2.13} \\
&&+h^{\bar{m}k}(\frac{d\zeta _{\bar{m}}}{ds}-N_{\bar{r}\bar{m}}\frac{d\bar{z}%
^{r}}{ds})+\Theta ^{\ast k}.  \notag
\end{eqnarray}

Now, using (\ref{2.11}) and (\ref{2}), (\ref{2.13}) leads to%
\begin{equation}
h^{kl}(\frac{d\zeta _{l}}{ds}-N_{jl}\zeta ^{j}-h_{jl}\theta ^{\ast j})+h^{%
\bar{m}k}(\frac{d\zeta _{\bar{m}}}{ds}-N_{\bar{r}\bar{m}}\frac{d\bar{z}^{r}}{%
ds}-h_{\bar{r}\bar{m}}\Theta ^{\ast \bar{r}})=0.  \label{2.14}
\end{equation}

Denoting with $S_{l}:=\frac{d\zeta _{l}}{ds}-N_{jl}\zeta ^{j}-h_{jl}\theta
^{\ast j}$, the equation (\ref{2.14}) becomes%
\begin{equation}
h^{kl}S_{l}+h^{\bar{m}k}S_{\bar{m}}=0,  \label{2.15}
\end{equation}%
which yields $(h^{\bar{k}\bar{l}}h^{im}h_{i\bar{l}}-h^{\bar{k}m})S_{m}=0.$
Since $h^{\bar{k}\bar{l}}h^{im}h_{i\bar{l}}=0,$ then $h^{\bar{k}m}S_{m}=0,$
and so, $S_{m}=0$, i.e., $\sigma ^{\ast }(s)$ is a complex geodesic curve
for $\mathcal{C}$.
\end{proof}

\subsection{Projectively related complex Cartan metrics}

In \cite{Al-Mu2} we investigated the projectively related complex Finsler
metrics. Namely, the complex Finsler metrics $F$ and $\tilde{F}$ are called
projectively related if they have the same geodesic curves as point sets.
This means that for any complex geodesic curves: $\sigma _{1}=\sigma _{1}(s)$
of $(M,F),$ (given by (\ref{2.1}) or equivalently (\ref{2.2})), and $\sigma
_{2}=\sigma _{2}(\tilde{s})$ \ of $(M,\tilde{F}),$ (given by $\frac{%
d^{2}z^{k}}{d\tilde{s}^{2}}+2\tilde{G}^{k}(z(\tilde{s}),\frac{dz}{d\tilde{s}}%
)=\tilde{\Theta}^{k}(z(\tilde{s}),\frac{dz}{d\tilde{s}})),$ then $\sigma
_{1} $ and $\sigma _{2}$ must represent the same set of points. To achieve
this, we compare the above mentioned equations, making the same parameter $%
t. $ The equations of geodesic curve $\sigma _{1}$ in the parameter $t(s),$
with $\frac{dt}{ds}>0,$ are not preserved because it is transformed in%
\begin{equation}
\lbrack T^{k}(t)-\Theta ^{k}(t)]\left( \frac{dt}{ds}\right)
^{2}=T^{k}(s)-\Theta ^{k}(s)-\frac{dz^{k}}{dt}\frac{d^{2}t}{ds^{2}}=-\frac{%
dz^{k}}{dt}\frac{d^{2}t}{ds^{2}},  \label{5.1}
\end{equation}%
where $T^{k}(t):=\frac{d^{2}z^{k}}{dt^{2}}+N_{l}^{k}(t)\frac{dz^{l}}{dt},$
with $N_{l}^{k}(t):=N_{l}^{k}(z,\frac{dz}{dt})$ and $\Theta ^{k}(t):=\Theta
^{k}(z(t),\frac{dz}{dt})$, (for more details see \cite{Al-Mu2}).

Similar equations are obtained for $\sigma _{2}(t(\tilde{s})).$ Subtracting
these equations, we obtained that the spray coefficients of two projectively
related complex Finsler metrics are linked by $\tilde{G}^{k}=G^{k}+B^{k}+P%
\eta ^{k},$ where $B^{k}=\frac{1}{2}(\tilde{\Theta}^{k}-\Theta ^{k})$ and $%
P(z,\eta )$ is a smooth function on $\widetilde{T^{\prime }M}.$\medskip

Based on these, we introduce by $\mathcal{L}-$ duality the corespondent
notion on the $\mathcal{L}-$ dual complex Cartan spaces.

Let $M$ $\ $be the complex manifold with $F$ and $\tilde{F}$ projectively
related Finsler metrics. By (\ref{4.2}), we obtain two Cartan metrics $%
\mathcal{C}$ and $\mathcal{\tilde{C}},$ which are the images by $\mathcal{L}%
- $ duality of $F$ and $\tilde{F},$ respectively. \bigskip Having in mind
the notion of complex geodesic curve on $(M,\mathcal{\tilde{C}})$ as being a
$\mathcal{L}-$ dual of a complex geodesic curve on $(M,F),$ introduced in
the preview section, we give.

\begin{definition}
The complex Cartan metrics $\mathcal{C}$ and $\mathcal{\tilde{C}}$ on the
manifold $M$, which are the images by $\mathcal{L}-$ duality of $\ $the
complex Finsler metrics $F$ and $\tilde{F},$ respectively, are called
projectively related if they have the same complex geodesic curves as point
sets.
\end{definition}

Our next goal is to find the image by $\mathcal{L}-$ duality of the
equations (\ref{5.1}). For this, a similar calculus as in (\ref{2.5}), yields%
\begin{eqnarray*}
T^{k}(s) &=&-g^{\bar{n}k}[\mathring{N}_{\bar{n}\bar{r}}(s)\frac{d\bar{z}^{r}%
}{ds}+g_{\bar{r}\bar{n}}T^{\bar{r}}(s)-\frac{d\eta _{\bar{n}}}{ds}] \\
&=&-g^{\bar{n}k}[\mathring{N}_{\bar{n}\bar{r}}(t)\frac{d\bar{z}^{r}}{dt}+g_{%
\bar{r}\bar{n}}T^{\bar{r}}(t)-\frac{d\eta _{\bar{n}}}{dt}]\left( \frac{dt}{ds%
}\right) ^{2} \\
&=&T^{k}(t)\left( \frac{dt}{ds}\right) ^{2},
\end{eqnarray*}%
which together with (\ref{5.1}) leads to%
\begin{equation}
-g^{\bar{n}k}[\mathring{N}_{\bar{n}\bar{r}}(t)\frac{d\bar{z}^{r}}{dt}+g_{%
\bar{r}\bar{n}}T^{\bar{r}}(t)-\frac{d\eta _{\bar{n}}}{dt}]-\Theta ^{k}(t)=-%
\frac{dz^{k}}{dt}\frac{d^{2}t}{ds^{2}}\frac{1}{\left( \frac{dt}{ds}\right)
^{2}}.  \label{5.2}
\end{equation}%
Now, setting the image by $\mathcal{L}-$ duality of the equation (\ref{5.2}%
), it results the equations of the complex geodesic curve $\sigma _{1}^{\ast
}=\sigma _{1}^{\ast }(t(s))$ in the parameter $t$%
\begin{eqnarray*}
&&-h^{\bar{n}k}\{N_{\bar{n}\bar{r}}\frac{d\bar{z}^{r}}{dt}+h_{\bar{r}\bar{n}%
}[\Theta ^{\ast \bar{r}}(t)-\frac{d\bar{z}^{r}}{dt}\frac{d^{2}t}{ds^{2}}%
\frac{1}{\left( \frac{dt}{ds}\right) ^{2}}]-\frac{d\zeta _{\bar{n}}}{dt}%
\}-\Theta ^{\ast k}(t) \\
&=&-\frac{dz^{k}}{dt}\frac{d^{2}t}{ds^{2}}\frac{1}{\left( \frac{dt}{ds}%
\right) ^{2}}.
\end{eqnarray*}%
and $\frac{dz^{k}}{dt}=[\eta ^{k}(t)]^{\ast }=\zeta ^{k}(t),$ which are
equivalent with%
\begin{equation}
\frac{d\zeta _{\bar{n}}}{dt}-N_{\bar{r}\bar{n}}\frac{d\bar{z}^{r}}{dt}=h_{%
\bar{r}\bar{n}}\Theta ^{\ast \bar{r}}(t)-(h_{\bar{r}\bar{n}}\frac{d\bar{z}%
^{r}}{dt}+h_{k\bar{n}}\frac{dz^{k}}{dt})\frac{d^{2}t}{ds^{2}}\frac{1}{\left(
\frac{dt}{ds}\right) ^{2}}  \label{14}
\end{equation}%
and $\frac{dz^{k}}{dt}=\zeta ^{k}(t).$ Taking the conjugation of (\ref{14})
and then using $h_{jk}\eta ^{j}=h_{jk}\frac{dz^{j}}{dt}=0$ (from Lemma 5.2),
we obtain that the equations (\ref{5.2}) in parameter $t$ are%
\begin{equation}
\frac{d\zeta _{k}}{dt}-N_{jk}\zeta ^{j}(t)-h_{jk}\Theta ^{\ast j}(t)=-\zeta
_{k}\frac{d^{2}t}{ds^{2}}\frac{1}{\left( \frac{dt}{ds}\right) ^{2}},\;k=%
\overline{1,n},  \label{5.3}
\end{equation}%
and $\zeta ^{k}(t)=\frac{dz^{k}}{dt}.$

Note that, by the transformation of the parameter $t=t(s),$\ with $\frac{dt}{%
ds}>0$, the equations of (\ref{5.3})\textbf{\ }are not preserved.

Corresponding to the complex Cartan metric $\mathcal{\tilde{C}},$ on the
same manifold $M,$ we have the coefficients $\tilde{N}_{jk}$ of Chern-Cartan
(c.n.c.), $\tilde{\zeta}^{j}:=\frac{\partial \tilde{H}}{\partial \zeta _{k}}$
and the functions $\tilde{\Theta}^{\ast k}.$ Also, we suppose that $\sigma
_{2}^{\ast }=\sigma _{2}^{\ast }(\tilde{s})$ is a complex geodesic curve of $%
(M,\mathcal{\tilde{C}}),$ where $\tilde{s}$ is the parameter corresponding
to $\mathcal{\tilde{C}}$. Now, assuming that the same parameter $t$ is
transformed by $t=t(\tilde{s})$ as above, we obtain%
\begin{equation}
\frac{d\zeta _{k}}{dt}-\tilde{N}_{jk}\tilde{\zeta}^{j}(t)-\tilde{h}_{jk}%
\tilde{\Theta}^{\ast j}(t)=-\zeta _{k}\frac{d^{2}t}{d\tilde{s}^{2}}\frac{1}{%
\left( \frac{dt}{d\tilde{s}}\right) ^{2}},\;k=\overline{1,n}.  \label{5.4}
\end{equation}

If $\mathcal{C}$ and $\mathcal{\tilde{C}}$ are projectively related, then $%
\sigma _{1}^{\ast }$ and $\sigma _{2}^{\ast }$ represent the same set of
points and, the difference between corresponding equations from (\ref{5.3})
and (\ref{5.4}) gives%
\begin{eqnarray}
&&\tilde{N}_{jk}\tilde{\zeta}^{j}+\tilde{h}_{jk}\tilde{\Theta}^{\ast
j}-N_{jk}\zeta ^{j}-h_{jk}\Theta ^{\ast j}  \label{4.6} \\
&=&\zeta _{k}[\frac{d^{2}t}{d\tilde{s}^{2}}\frac{1}{\left( \frac{dt}{d\tilde{%
s}}\right) ^{2}}-\frac{d^{2}t}{ds^{2}}\frac{1}{\left( \frac{dt}{ds}\right)
^{2}}],\;k=\overline{1,n}.  \notag
\end{eqnarray}

With the notations: $\tilde{N}_{k}:=\tilde{N}_{jk}\tilde{\zeta}^{j}$ and $%
N_{k}:=N_{jk}\zeta ^{j},$ (\ref{4.6}) can be rewritten more generally as%
\begin{equation}
\tilde{N}_{k}+\tilde{h}_{jk}\tilde{\Theta}^{\ast j}=N_{k}+h_{jk}\Theta
^{\ast j}+Q\zeta _{k},\;k=\overline{1,n},  \label{4.7}
\end{equation}%
where $Q$ is a smooth function on $\widetilde{T^{\prime \ast }M}$, with
complex values.

Denoting by $B_{k}:=\tilde{h}_{jk}\tilde{\Theta}^{\ast j}-h_{jk}\Theta
^{\ast j},$ the homogeneity properties of the functions $h_{jk}$ and $\Theta
^{\ast j}$ give $(\dot{\partial}^{k}B_{i})\zeta _{k}=2B_{i}$ and $(\dot{%
\partial}^{\bar{k}}B_{i})\bar{\zeta}_{k}=0.$ Moreover, the relations (\ref%
{4.7}) become%
\begin{equation}
\tilde{N}_{k}=N_{k}+B_{k}+Q\zeta _{k}.  \label{5}
\end{equation}

Now, we use their homogeneity properties, going from $\zeta _{k}$ to $%
\lambda \zeta _{k}.$ Thus, differentiating in (\ref{5}) with respect to $%
\zeta _{k}$ and $\bar{\zeta}_{k}$ and then setting $\lambda =1$, we obtain%
\begin{equation}
B_{k}=-(\dot{\partial}^{r}Q)\zeta _{r}\zeta _{k}\;\;\;\mbox{and}\;\;B_{k}=[(%
\dot{\partial}^{\bar{r}}Q)\zeta _{\bar{r}}-Q]\zeta _{k}  \label{6}
\end{equation}%
and so,%
\begin{equation}
(\dot{\partial}^{r}Q)\zeta _{r}+(\dot{\partial}^{\bar{r}}Q)\zeta _{\bar{r}%
}=Q,  \label{7}
\end{equation}%
which means that $Q(z^{k},\mu \zeta _{k})=\mu Q(z^{k},\zeta _{k})$, for any $%
\mu \in \mathbf{R}$.

\begin{lemma}
Between the coefficients $\tilde{N}_{k}$ and $N_{k}$ corresponding to the
metrics $\mathcal{C}$ and $\mathcal{\tilde{C}}$ on the manifold $M$ there
are the relations $\tilde{N}_{k}=N_{k}+B_{k}+Q\zeta _{k},$ for any $k=%
\overline{1,n},$ where $Q$ is a smooth function on $\widetilde{T^{\prime
\ast }M}$ with complex values, if and only if $\tilde{N}_{k}=N_{k}+(\dot{%
\partial}^{\bar{r}}Q)\zeta _{\bar{r}}\zeta _{k}$, $B_{k}=-(\dot{\partial}%
^{r}Q)\zeta _{r}\zeta _{k}$, for any $k=\overline{1,n},$ and $(\dot{\partial}%
^{r}Q)\zeta _{r}+(\dot{\partial}^{\bar{r}}Q)\zeta _{\bar{r}}=Q$.
\end{lemma}

From above considerations, we obtain.

\begin{lemma}
If the complex Cartan metrics $\mathcal{C}$ and $\mathcal{\tilde{C}}$ on the
manifold $M$ are projectively related, then there is a smooth function $Q$
on $\widetilde{T^{\prime \ast }M}$ with complex values, satisfying $(\dot{%
\partial}^{r}Q)\zeta _{r}+(\dot{\partial}^{\bar{r}}Q)\zeta _{\bar{r}}=Q,$
such that%
\begin{equation}
\tilde{N}_{k}=N_{k}+(\dot{\partial}^{\bar{r}}Q)\zeta _{\bar{r}}\zeta _{k}\;%
\mbox{and}\;B_{k}=-(\dot{\partial}^{r}Q)\zeta _{r}\zeta _{k}\;;\;k=\overline{%
1,n}.  \label{8}
\end{equation}
\end{lemma}

Conversely, under assumption that $\sigma ^{\ast }=\sigma ^{\ast }(s)$ is a
complex geodesic curve of $(M,\mathcal{C}),$ we show that the complex Cartan
metric $\mathcal{\tilde{C}}$ with the coefficients $\tilde{N}_{jk}$ of
Chern-Cartan (c.n.c.) and the functions $\tilde{\zeta}^{k}=\frac{\partial
\tilde{H}}{\partial \zeta _{k}},$ given by (\ref{4.7}) is projectively
related to $\mathcal{C}$, where $Q$ is a smooth function on $\widetilde{%
T^{\prime \ast }M}$ with complex values. This means that there is a
parametrization $\tilde{s}=\tilde{s}(s),$ with $\frac{d\tilde{s}}{ds}>0,$
such that $\sigma ^{\ast }=\sigma ^{\ast }(\tilde{s}(s))$ is a geodesic of $%
(M,\mathcal{\tilde{C}}).$

If there is a parametrization $\tilde{s}=\tilde{s}(s),$ then it yields

$\frac{d\zeta _{k}}{d\tilde{s}}-N_{rk}\zeta ^{r}(\tilde{s})-h_{rk}\Theta
^{\ast r}(\tilde{s})=-\zeta _{k}\frac{d^{2}\tilde{s}}{ds^{2}}\frac{1}{\left(
\frac{d\tilde{s}}{ds}\right) ^{2}},\;k=\overline{1,n}.$

Now, using (\ref{4.7}), it results%
\begin{equation*}
\frac{d\zeta _{k}}{d\tilde{s}}=\tilde{N}_{k}(\tilde{s})+\tilde{h}_{rk}\tilde{%
\Theta}^{\ast r}(\tilde{s})+[Q(\tilde{s})-\frac{d^{2}\tilde{s}}{ds^{2}}\frac{%
1}{\left( \frac{d\tilde{s}}{ds}\right) ^{2}}]\zeta _{k}\;;\;k=\overline{1,n}.
\end{equation*}%
So, $\sigma ^{\ast }=\sigma ^{\ast }(\tilde{s}(s))$ is a geodesic of $(M,%
\mathcal{\tilde{C}})$ if and only if%
\begin{equation}
\lbrack Q(\tilde{s})-\frac{d^{2}\tilde{s}}{ds^{2}}\frac{1}{\left( \frac{d%
\tilde{s}}{ds}\right) ^{2}}]\zeta _{k}=0\;;\;k=\overline{1,n}.  \label{9'}
\end{equation}%
Since $\zeta _{k}\neq 0$, it results $Q(\tilde{s})\left( \frac{d\tilde{s}}{ds%
}\right) ^{2}=\frac{d^{2}\tilde{s}}{ds^{2}}.$ Due to (\ref{7}), it leads to%
\begin{equation}
Q(s)\frac{d\tilde{s}}{ds}=\frac{d^{2}\tilde{s}}{ds^{2}}.  \label{10'}
\end{equation}%
Denoting by $u(s):=\frac{d\tilde{s}}{ds}$, we have $\frac{d^{2}\tilde{s}}{%
ds^{2}}=\frac{du}{ds}$ and so, $Q(s)u=\frac{du}{ds}.$ We obtain $u=ae^{\int
Q(s)ds}.$ From here, it results that there is%
\begin{equation*}
\tilde{s}(s)=a\int e^{\int Q(s)ds}ds+b,
\end{equation*}%
where $a,b$ are arbitrary constants.

Corroborating all above results we have proven.

\begin{theorem}
Let $\mathcal{C}$ and $\mathcal{\tilde{C}}$ be complex Cartan metrics on the
manifold $M$. Then $\mathcal{C}$ and $\mathcal{\tilde{C}}$ are projectively
related if and only if there is a smooth function $Q$ on $\widetilde{%
T^{\prime \ast }M}$ with complex values, such that%
\begin{equation}
\tilde{N}_{k}=N_{k}+B_{k}+Q\zeta _{k};\;k=\overline{1,n}.  \label{12}
\end{equation}
\end{theorem}

As a consequence of Lemma 5.3 we have the following.

\begin{corollary}
Let $\mathcal{C}$ and $\mathcal{\tilde{C}}$ be the complex Cartan metrics on
the manifold $M$. $\mathcal{C}$ and $\mathcal{\tilde{C}}$ are projectively
related if and only if there is a smooth function $Q$ on $\widetilde{%
T^{\prime \ast }M}$ with complex values, such that $\tilde{N}_{k}=N_{k}+(%
\dot{\partial}^{\bar{r}}Q)\zeta _{\bar{r}}\zeta _{k}$, $B_{k}=-(\dot{\partial%
}^{r}Q)\zeta _{r}\zeta _{k}$, for any $k=\overline{1,n},$ and $(\dot{\partial%
}^{r}Q)\zeta _{r}+(\dot{\partial}^{\bar{r}}Q)\zeta _{\bar{r}}=Q$.
\end{corollary}

The relations (\ref{12}) between the functions $\tilde{N}_{k}$ and $N_{k}$
of the projectively related complex Cartan metrics $\mathcal{C}$ and $%
\mathcal{\tilde{C}}$ will be called \textit{projective change.}

\begin{theorem}
Let $\mathcal{C}$ and $\mathcal{\tilde{C}}$ be the complex Cartan metrics on
the manifold $M,$ which are projectively related. If either $\mathcal{C}$ or
$\mathcal{\tilde{C}}$ is weakly K\"{a}hler-Cartan then, $B_{k}=0$ and the
projective change is $\tilde{N}_{k}=N_{k}+Q\zeta _{k}$, where $Q$ is a $%
(0,1) $ - homogeneous function.
\end{theorem}

\begin{proof}  We assume that $\tilde{N}_{k}=N_{k}+(\dot{\partial}^{\bar{r}}Q)\zeta
_{\bar{r}}\zeta _{k}$, $B_{k}=-(\dot{\partial}^{r}Q)\zeta _{r}\zeta _{k}$
and $(\dot{\partial}^{r}Q)\zeta _{r}+(\dot{\partial}^{\bar{r}}Q)\zeta _{\bar{%
r}}=Q$.

If $\mathcal{C}$ is weakly K\"{a}hler then $\Theta ^{\ast s}=0$ and so, $%
\tilde{h}_{sk}\tilde{\Theta}^{\ast s}$ $=-(\dot{\partial}^{r}Q)\zeta
_{r}\zeta _{k},$ which contracted by $\tilde{\zeta}^{k},$ implies $(\dot{%
\partial}^{r}Q)\zeta _{r}=0.$ This leads to $B_{k}=\tilde{h}_{sk}\tilde{%
\Theta}^{\ast s}=0$.
\end{proof}

\begin{theorem}
Let $\mathcal{C}$ be complex Euclidean metric on a domain $D$ from $\mathbf{C%
}^{n}$ and $\mathcal{\tilde{C}}$ another complex Cartan metric on $D$. Then $%
\mathcal{C}$ and $\mathcal{\tilde{C}}$ are projectively related if and only
if $\tilde{N}_{k}=Q\zeta _{k}$ and $\tilde{h}_{sk}\tilde{\Theta}^{\ast s}=0,$%
where $Q=-\frac{1}{\tilde{H}}\frac{\partial ^{\ast }\tilde{H}}{\partial z^{i}%
}\tilde{\zeta}^{i}.$
\end{theorem}

\begin{proof}
We suppose that $\mathcal{C}$ and $\mathcal{\tilde{C}}$ are projectively
related. But, the complex Euclidean metric $\mathcal{C}:=|\zeta |^{2}={\sum }%
_{k=1}^{n}\zeta _{k}\zeta _{\bar{k}}$ is K\"{a}hler, (it has $\Theta ^{\ast
s}=0)$ with $N_{k}=0$. By these assumptions and taking into account Theorem
5.6 it results $\tilde{N}_{k}=Q\zeta _{k}$ and $\tilde{h}_{sk}\tilde{\Theta}%
^{\ast s}=0.$ Further on, contracting with $\tilde{\zeta}^{k}$ the relation $%
\tilde{N}_{k}=Q\zeta _{k}$ we obtain $Q=-\frac{1}{\tilde{H}}\frac{\partial
^{\ast }\tilde{H}}{\partial z^{i}}\tilde{\zeta}^{i}.$ The converse is
obvious.
\end{proof}
This section is completed with two applications.

We first wish to find if a complex Cartan-Randers metric $\mathcal{\tilde{C}}%
=\alpha +|\beta |$ can be obtained by $\mathcal{L}-$ duality from a complex
Finsler metric. We do not expect the metric\textbf{\ }$\mathcal{\tilde{C}}%
=\alpha +|\beta |$\textbf{\ }comes from a complex\textbf{\ }Finsler-Randers%
\textbf{\ }metric.\textbf{\ }Below, we show a degree more.\textbf{\ }The
existence of the nonholonomic vertical frame $\{\frac{\partial }{\partial
\zeta ^{k}},\frac{\partial }{\partial \bar{\zeta}^{k}}\}$ is conditioned by
the relations (\ref{2}). Thus, we can prove the following result.

\begin{theorem}
Let $(M,\mathcal{\tilde{C}})$ be a complex Cartan-Randers space. Then, $%
\mathcal{\tilde{C}}$ is image by $\mathcal{L}-$ duality of a complex Finsler
metric on the same manifold $M$ if and only if $\mathcal{\tilde{C}}$ is a
purely Hermitian complex Cartan-Randers metric.
\end{theorem}

\begin{proof}  Corresponding to the metric $\mathcal{\tilde{C}}=\alpha +|\beta |$,
by direct computation it results

$\tilde{h}^{ji}=\frac{\partial ^{2}\mathcal{\tilde{C}}^{2}}{\partial \zeta
_{j}\partial \zeta _{i}}=-\frac{\alpha |\beta |}{2}(\frac{1}{\alpha ^{2}}%
\zeta ^{i}-\frac{\bar{\beta}}{|\beta |^{2}}b^{i})(\frac{1}{\alpha ^{2}}\zeta
^{j}-\frac{\bar{\beta}}{|\beta |^{2}}b^{j}).$

Now, we suppose that $\mathcal{\tilde{C}}$ is image by $\mathcal{L}-$
duality of a complex Finsler metric. Then, $\mathcal{\tilde{C}}$ must
satisfy (\ref{2}). The second relation from (\ref{2}) leads to $\tilde{h}%
_{ij}=\frac{2\alpha |\beta |}{\gamma ^{2}\mathcal{\tilde{C}}^{2}}[(\alpha
||b||^{2}+|\beta |)\zeta _{i}-\frac{\alpha \mathcal{\tilde{C}}\beta }{|\beta
|}b_{i}][(\alpha ||b||^{2}+|\beta |)\zeta _{j}-\frac{\alpha \mathcal{\tilde{C%
}}\beta }{|\beta |}b_{j}].$ Thus, the first condition from (\ref{2}) gives $%
\alpha ^{2}||b||^{2}=|\beta |^{2},$ i.e., $\mathcal{\tilde{C}}$ is a purely
Hermitian metric.

The converse is obvious. Indeed, if $\mathcal{\tilde{C}}$ is a purely
Hermitian metric then $\tilde{h}^{ji}=\tilde{h}_{ij}=0$, which proves our
claim.
\end{proof}

Therefore, we can to discuss only the projectiveness of the purely Hermitian
complex Cartan-Randers metric $\mathcal{\tilde{C}}=(1+||b||)\alpha .$

Since $\mathcal{\tilde{C}}$ is a complex Berwald-Cartan metric, then $\tilde{%
N}_{ji}=\overset{a}{N_{ji}}$ which contracted with $\tilde{\zeta}%
^{j}:=(1+||b||)^{2}\zeta ^{j}$ implies%
\begin{equation}
\tilde{N}_{i}=\overset{a}{N_{i}}(1+||b||)^{2}.  \label{ii}
\end{equation}
So, we obtain.

\begin{corollary}
The metrics $\mathcal{\tilde{C}}=(1+||b||)\alpha $ $\ $\ and $\alpha $ are
projectively related if and only if $\overset{a}{N_{i}}=P\zeta _{i},$ where $%
P=-\frac{(2+||b||)||b||}{\alpha ^{2}}\frac{\partial \alpha ^{2}}{\partial
z^{i}}\zeta ^{i}.$
\end{corollary}

\begin{proof}  We suppose that $\mathcal{\tilde{C}}=(1+||b||)\alpha $ $\ $\ and $%
\alpha $ are projectively related. Since $\mathcal{\tilde{C}}$ $\ $\ and $%
\alpha $ are purely Hermitian, then there is a smooth function $Q$ on $%
\widetilde{T^{\prime \ast }M}$ with complex values, such that $\tilde{N}_{i}=%
\overset{a}{N_{i}}+Q\zeta _{i}.$ But, using (\ref{ii}), it results $Q\zeta
_{i}=(2+||b||)||b||\overset{a}{N_{i}}.$ Contracting the last relation with $%
\zeta ^{i},\ $we obtain $Q=\frac{(2+||b||)||b||}{\alpha ^{2}}\overset{a}{%
N_{i}}\zeta ^{i}=-\frac{(2+||b||)||b||}{\alpha ^{2}}\frac{\partial \alpha
^{2}}{\partial z^{i}}\zeta ^{i},$ which prove the direct claim.

Conversely, if $\overset{a}{N_{i}}=P\zeta _{i}$ then taking into account (%
\ref{ii}), we obtain

$\tilde{N}_{i}=\overset{a}{N_{i}}+(2+||b||)||b||\overset{a}{N_{i}}=\overset{a%
}{N_{i}}+(2+||b||)||b||P\zeta _{i}=\overset{a}{N_{i}}+Q\zeta _{i},$ where $%
Q:=(2+||b||)||b||P=-\frac{(2+||b||)^{2}||b||^{2}}{\alpha ^{2}}\frac{\partial
\alpha ^{2}}{\partial z^{i}}\zeta ^{i}$ is a smooth function $Q$ on $%
\widetilde{T^{\prime \ast }M}$ with complex values. This completes our proof.
\end{proof}

Moreover, the metric $\mathcal{\tilde{C}}=(1+||b||)\alpha $ is projectively
related with the complex Euclidean metric $\mathcal{C}$ on a domain $D$ if
and only if $\alpha $ is projectively related with $\mathcal{C}.$\bigskip

The second application refers to the locally projectively flat complex
Cartan metrics. Let $\mathcal{\tilde{C}}$ be a locally Minkowski complex
Cartan metric on the underlying manifold $M.$ Corresponding to the metric $%
\mathcal{\tilde{C}}$ there are exist in any point the local charts in which
we have $\tilde{N}_{k}=0$ and $\tilde{\Theta}^{\ast s}=0$ because, in such
local charts, the fundamental metric tensor $\tilde{h}^{\bar{m}i}$ depends
only on $\zeta $.

Also, we consider $\mathcal{C}$ another complex Cartan metric on the complex
manifold $M$. Note that, we have assumed that $\mathcal{C}$ and $\mathcal{%
\tilde{C}}$ are the images by $\mathcal{L}-$ duality of the locally regulate
complex Finsler metrics $F$ and $\tilde{F}$ on $M,$ respectively.

The complex Finsler metrics $\mathcal{C}$ will be called \textit{locally
projectively flat} if it is projectively related to the locally Minkowski
metric $\mathcal{\tilde{C}}.$

\begin{corollary}
$\mathcal{C}$ is locally projectively flat if and only if $h_{sk}\Theta
^{\ast s}=0$ and $N_{k}=-Q\zeta _{k},$where $Q=\frac{2}{\mathcal{C}}\frac{%
\partial ^{\ast }C}{\partial z^{i}}\zeta ^{i}.$
\end{corollary}

\begin{proof} It follows by Theorems 5.5 and 5.6.
\end{proof}

\begin{flushleft}
Transilvania Univ., Faculty of Mathematics and Informatics

Iuliu Maniu 50, Bra\c{s}ov 500091, ROMANIA

e-mail: nicoleta.aldea@lycos.com

e-mail: gh.munteanu@unitbv.ro
\end{flushleft}

\end{document}